\documentclass[11pt]{article}
\usepackage{geometry}                % See geometry.pdf to learn the layout options. There are lots.
\usepackage{graphicx}
\usepackage{amssymb}
\usepackage{amsmath}
\usepackage{epstopdf}
\usepackage{amsthm}
\usepackage{threeparttable}
\usepackage{booktabs}
\usepackage{setspace}
\usepackage[ruled,vlined]{algorithm2e}
\newtheorem{lemma}{Lemma}
\newtheorem{theorem}{Theorem}

\newtheorem{definition}{Definition}

\newtheorem{assumption}{Assumption}
\newtheorem{observation}{Observation}

\def\mr{\mathbb{R}}

\def\lj38{{\rm LJ}_{38}}
\def\mst{\mathcal{T}^{\ast}}

 \title{Computing the Asymptotic Spectrum \\ for Networks Representing Energy Landscapes \\
using the Minimal Spanning Tree}
 
 \author{Maria Cameron$^1$ }

\begin{document}
\maketitle

 \footnotetext[1]{University of Maryland, 
 Department of Mathematics, 
 College Park, MD  20742,
 {\tt cameron@math.umd.edu}}

\abstract{The concept of metastability has
caused a lot of interest in recent years. 
The spectral decomposition of the generator matrix of a
stochastic network exposes all of the transition processes in the system.
The assumption of the existence of a low lying group of eigenvalues separated by a spectral gap, leading to 
factorization of the dynamics, has become a popular theme.
We consider stochastic networks representing potential energy landscapes where the 
states and the edges correspond to local minima and  transition states 
 respectively, 
and the pairwise transition rates are given by the Arrhenuis formula.
Using the minimal spanning tree, we construct the asymptotics for  eigenvalues
and  eigenvectors of the generator matrix starting from the low lying group. 
This construction gives rise to an
efficient algorithm for computing the asymptotic spectrum
suitable for large and complex networks.
We apply it to Wales's Lennard-Jones-38 network with 71887 states and 119853 edges where 
the underlying potential energy landscape has a double-funnel structure. Our results
demonstrate that the concept of metastability should be applied with care to this system.
For the full network, there is no significant spectral gap separating
the eigenvalue corresponding to the exit from the wider and shallower
icosahedral funnel at any reasonable temperature range.
However, if the observation time is limited, the expected spectral gap appears.

}
\section{Introduction}
\label{sec:intro}
In this work we consider stochastic networks with detailed balance where the pairwise transition rates are of the form
\begin{equation}
\label{eq1}
L_{ij}=\begin{cases}\frac{k_{ij}}{k_i}e^{-(V_{ij}-V_i)/T},&~~{\rm if}~i\sim j,\\ 0,& {\rm otherwise},\end{cases},\quad{\rm where}\quad i\neq j.
\end{equation}
Networks of this kind represent, e.g., potential energy landscapes where all critical points are isolated.  
The set of states  is equivalent to the  set of local minima, and the set of  edges 
is equivalent to the set of transition states or Morse index one saddles separating the local minima. 
States $i$ and $j$  are connected by an edge $(i,j)$ (notation $i\sim j$)  if and only if the
corresponding  local minima are separated by a single saddle\footnotemark[1].   The number 
$V_{ij}$ in Eq. \eqref{eq1} is the potential at the saddle $ij$ separating $i$ and $j$, $V_i$ is the potential at the minimum $i$, $k_{ij}$ and $k_i$ are
temperature-independent prefactors defined by the Hessian matrices  and the orders of the point groups 
of the saddle $ij$ and the minimum $i$ respectively \cite{wales0}.
$T$ is the temperature, a small parameter.
Eq. \eqref{eq1} defines the off-diagonal entries of the  generator matrix $L$ while its diagonal entries are defined so that the sum of entries in each row is zero, i.e.,
\begin{equation}
\label{eq2}
L_{ii}=-\sum_{j\neq i}L_{ij}.
\end{equation}

\footnotetext[1]{This criterion for the states being connected by an edge can be relaxed.
More generally, we connect states $i$ and $j$ by an edge $(i,j)$  if and only if 
one can find a Minimum Energy Path $\phi_{ij}(\alpha)$, $\alpha\in[0,1]$ with
the following properties:
$(i)$ $\phi_{ij}(0)=x_1$ and $\phi_{ij}(1)=x_j$, where $x_i$ and $x_j$ are the  local minima corresponding to the states $i$ and $j$;
$(ii)$ $\phi_{ij}$ passes through no other local minima other than its endpoints $x_i$ and $x_j$; 
$(iii)$ the only critical points that $\phi_{ij}$ passes through are saddles;
$(iv)$ the maximal value of the potential along $\phi_{ij}$  is achieved at a Morse index one saddle.
Then the number $V_{ij}$ is the maximal potential value along $\phi_{ij}$. 
A number of interesting phenomena regarding the Minimum Energy Paths is discussed in \cite{ckve}.
}

D. Wales \cite{wales0, wales1, wales_landscapes} proposed to model the low temperature dynamics of a molecular cluster
by the dynamics of the corresponding stochastic network.
Wales and his group developed efficient tools for generating and exploring stochastic networks representing energy landscapes.
A large collection of them can be found at the web site \cite{web}. 
Wales's stochastic networks are complex and fascinating. They exhibit metastability,
offer rich families of possible transition paths, and
involve a remarkable interplay between energetic and entropic barriers. Their study
evokes new theoretical paradigms  and inspires the development of new computational tools.

Another context where networks with pairwise transition rates of the form of Eq. \eqref{eq1}
arise is the evolutionary genetics. The networks represent
fitness landscapes in the models of evolutionary dynamics \cite{morozov,gillespie,kimura,ewens}.

Analysis of large stochastic networks is an interesting and challenging problem.
The number of states in the network representing an energy landscape  coming from chemical physics is   of the order of $10^p$, $p=3,4,5,6,\ldots$. 
The incidence matrix  is sparse but unstructured. The pairwise rates vary by tens of orders of magnitude.
Therefore, it is important to develop efficient computational tools able to cope with these difficulties.

One of the most appealing analysis tools of stochastic networks is the spectral decomposition of its generator matrix.
It reveals the whole collection of transition processes taking place in the system.
Originally, the asymptotics for the eigenvalues 
for of the generator matrices with entries of the order of $e^{-U_{ij}/T}$, without the assumption of the detailed balance,
was established by A. Wentzell \cite{wentzell1,wentzell2,f-w} in 1970s.  Wentzell's formulas,
involing optimization among the so called $W$-graphs, determine the whole collection of the eigenvalues 
up to the exponential order. 

In 2000s, Bovier and collaborators considered
systems with detailed balance and assumed the presence of 
 a spectral gap. 
They proved sharp estimates for low lying eigenvalues and the corresponding eigenvectors of Markov chains
with detailed balance in terms of capacities and exit times, and proposed 
a definition of metastability in terms of metastable points (representative points for metastable sets) \cite{bovier2002,bovier1,bovier2,bovier3}.

 Spectral analysis in the context of molecular systems was 
considered by Schuette and collaborators \cite{schuette03,schuette04}, and
another definition of metastability related to ergodicity  was proposed.  
An application of spectral analysis to clustering can be found in \cite{sarich}.

In this work, we focus on the construction of an efficient algorithm for computing the complete
asymptotic spectrum. Our starting point is Wentzell's formulas. 
We prove that in the case of detailed balance, 
the collection of the so called optimal $W$-graphs in Wentzell's formulas is nested and hence
can be built recursively 
starting from a certain minimal spanning tree and removing edges from it in a certain order. Then the
exponents determining the asymptotics of eigenvalues  as well  the asymptotics 
for eigenvectors are readily found from the optimal $W$-graphs.
These exponents also define exit rates from  certain Freidlin's cycles
 \cite{freidlin-cycles,freidlin-physicad,f-w,cam1} which are easily extracted from the 
optimal $W$-graphs as well. 
We propose a fast computational procedure for finding the collection of the optimal $W$-graphs and  the 
asymptotics for the full set of the eigenpairs 
starting from the smallest eigenvalues in the absolute value. 
Precisely, the output of the algorithm is the collection of potential differences $\Delta_k$ and sets $S_k$
such that the eigenvalues $\lambda_k$ are logarithmically equivalent to $\exp(-\Delta_k/T)$ and the eigenvectors 
$\phi_k$ are approximated by the indicator functions for the sets $S_k$. 
Modifying the stopping criterion, one can stop this algorithm
as soon as the eigenvalues exceed some provided threshold. 

Using our algorithm, we compute the asymptotic spectrum of Wales's stochastic network representing the Lennard-Jones
cluster of 38 atoms (we will refer to it as $\lj38$). 
The largest connected component of this network publicallly available via
Wales's group web site \cite{wales_network}
contains 71887 states and 119853 edges.
The $\lj38$ cluster is interesting because its potential energy landscape has a double-funnel structure
\cite{wales38,wales_landscapes}. The deeper and narrower funnel has the face-centered cubic truncated octahedron (FCC), the global minimum,
at the bottom, while
the shallower and wider funnel of icosahedral packings has the second lowest minimum (ICO)  at the bottom. These structures are show in Fig. \ref{fcc_ico}.
The double funnel feature might make us expect that the corresponding 
network is in some sense metastable.  Our results reveal that it is so in the sense of the definition by Schuette et al  \cite{schuette03,schuette04}
but not so in the sense of the definition of Bovier et al \cite{bovier2002,bovier1} at a reasonable range of 
temperatures. 
The reason is that this network has a large collection of local minima each of which is relatively high but separated
from the ground state by an even higher barrier. As a result, the set of the potential differences $\Delta_k$, $k=1,\ldots,n-1$, defining the exponents of the eigenvalues 
is relatively dense.
If the numbers $\Delta_k$ are ordered so that
$$\Delta_1\ge \Delta_2\ge\ldots\ge\Delta_{n-1},$$
(i.e., the corresponding eigenvalues are ordered according to their absolute values in the increasing order), 
the eigenvalue corresponding to exiting from
the icosahedral funnel is buried under the number 245. 
The gaps between the majority of the numbers $\Delta_k$, in particular, the gap $\Delta_{245}-\Delta_{246}$, 
are much smaller than the temperatures at which the $\lj38$
cluster is typically considered. 
Thus,  one cannot define, following Bovier et al, 
a set of  metastable points, one of which corresponds to ICO, satisfying the definition of the metastability.  
This means, that one cannot approximate the long-time dynamics of the $\lj38$ network by defining some number $m\ll n=71887$ 
of metastable sets and considering transitions between them. 
On the other hand, there is a large gap
between the number $\Delta_{245}$, determining the exit rate from the icosahedral basin, and the 
next largest $\Delta_k$ corresponding to a transition process within it. 
This means that if the system gets to the icosahedral basin, it will equilibrate there prior to exiting it.
Therefore,  the icosahedral basin is metastable in the sense of the definition by Schuette et al \cite{schuette03,schuette04}. 

We also would like to point out our use of disconnectivity graphs as a visualization tool.
Originally, they were introduced by Becker and  Karplus \cite{dgraphs} and extensively used by Wales et al  \cite{wales_dgraph,wales_book,wales_landscapes}.
Traditionally, the states are arranged along the $x$-axis arbitrarily, just so that the graph looks aesthetical.
We propose to organize the states along the $x$-axis according to some ordering of interest. 
In particular, this ordering can be by the number of the corresponding eigenvalue. In \cite{cve},
where the transition process between FCC and ICO was analyzed  at finite temperature,
we ordered states along the $x$-axis according to the committor (a. k. a.  the capacitor).

The rest of the paper is organized as follows. 
In Section \ref{sec:settings}, we provide a brief overview of some important properties of networks with detailed balance.
The theoretical relationships between the optimal $W$-graphs, Freidlin's cycles and the asymptotics of the spectrum 
are discussed in Section \ref{sec:theory}. The algorithm  for computing the asymptotics of the spectrum is introduced in Section \ref{sec:algorithms}.
The application to the $\lj38$ network is presented in Section \ref{sec:examples}.
We finish this paper with a conclusion in Section \ref{sec:conclusion}.

%%%%%%%%%%%%%%%%%%%%%%%%%%%%%

 \section{Spectral properties of networks with detailed balance}
\label{sec:settings}
We consider an irreducible network with a finite set of states $S$ and the generator matrix $L$ given by Eqs. \eqref{eq1}-\eqref{eq2}.
Eqs. \eqref{eq1}-\eqref{eq2} imply that the network possesses the detailed balance property
\begin{equation}
\label{detbal}
\pi_iL_{ij}=\pi_{j}L_{ji},
\end{equation}
where $\pi\equiv\{\pi_1,\pi_2,\ldots \pi_n\}$ is the equilibrium probability distribution satisfying
$$\pi^{T}L=0,\quad\sum_{i\in S}\pi_i=1.$$ 
The detailed balance condition \eqref{detbal} means that the expected numbers of transitions from state $i$ to state $j$ and vice versa per unit time are equal.

The detailed balance property dramatically simplifies the spectral analysis of the stochastic network.  
First, Eq. \eqref{detbal} implies that the generator matrix $L$ can be decomposed as 
\begin{equation}
\label{decomp}
L=P^{-1}Q,
\end{equation}
where $P=diag\{\pi_1,\pi_2,\ldots,\pi_n\}$, and $Q$ is  symmetric. 
Second, the eigenvalues of $L$ are real and nonpositive, and the eigenvectors of 
$L$ are orthogonal with respect to 
the inner $P$ product. 
These facts can be deduced from the similarity of  
$L=P^{-1}Q$ and the symmetric matrix $P^{-1/2}QP^{-1/2}$,
and the strict diagonal dominance of the matrix $(tI-L)$ for any $t>0$. 
The irreducibility of $L$ implies that the eigenvalue $0$ is simple.
We will write the matrix of eigenvalues of $L$ as 
\begin{equation}
\label{evals}
\Lambda:=diag\{0,-\lambda_1,-\lambda_2,\ldots,-\lambda_{n-1}\},\quad{\rm where}\quad
0<\lambda_1\le\lambda_2\le\ldots\le\lambda_{n-1}.
\end{equation}
Third, the eigen-decompositions of the matrices $L$ and $L^T$ can be written as 
\begin{equation}
\label{egg}
L=\Phi\Lambda\Phi^TP,\quad L^{T}=P\Phi\Lambda\Phi^{T}.
\end{equation}
In particular, since the row sums of $L$ are zeros, the eigenvector corresponding to the zero eigenvalue can be chosen to be $e:=[1,1,\ldots,1]^T$. 
The corresponding eigenvector of $L^T$ is $Pe\equiv \pi$, the equilibrium probability distribution.

The spectral decomposition of the stochastic network with detailed balance leads to a nice representation of the time
evolution of the probability distribution.
The probability distribution evolves according to the forward Kolmogorov (a. k. a.  the Fokker-Planck) equation
\begin{equation}
\label{eq:fk}
\frac{dp}{dt}=L^Tp,\quad p(0)=p_0.
\end{equation}
Using Eqs. \eqref{evals} and  \eqref{egg} one can write the solution of Eq. \eqref{eq:fk}  in the form
\begin{equation}
\label{eq:tev}
p(t)=e^{tL^T}p_0=P\Phi e^{t\Lambda }\Phi^{T}p_0 = \pi+\sum_{j=1}^{n-1} (\phi_j^Tp_0)P\phi_je^{-\lambda_jt},
\end{equation}
where $\Phi=[e,~\phi_1,~\ldots,~\phi_{n-1}]$. 
Eqs. \eqref{evals} and \eqref{eq:tev} show that, no matter what  the initial probability distribution $p(0)=p_0$ is,  it will evolve eventually
toward the equilibrium distribution $\pi$. However, the components $(\phi_j^Tp_0)P\phi_je^{-\lambda_jt}$  of $p(t)$ 
with small decay rates $\lambda_j$ can remain significant for long times, $O(\lambda_j^{-1})$. 
If the temperature is sufficiently small,  
the eigenvalues $\lambda_j$ of $-L$ are logarithmically equivalent to $\exp(-\Delta_j)/T$, 
where $\Delta_j$ are the certain constants determined by the values $V_{kl}$ and $V_i$, $i,k,l,\in S$ \cite{wentzell1,wentzell2,f-w}. 
Therefore, if the temperature is small enough and all numbers $\Delta_k$ are 
distinct, then
$$0<\lambda_1\ll\lambda_2\ll\ldots\ll\lambda_{n-1}.$$

%%%%%%%%%%%%%%%%%%%%%%%%%%%%%%%%%%%%%%%%
%%%%%%%%%%%%%%%%%%%%%%%%%%%%%%%%%%%%%%%%

\section{The spectrum,  the minimal spanning tree, and Freidlin's cycles}
\label{sec:theory}
In this Section, we present a construction that allows us to calculate the asymptotics for the eigenvalues and eigenvectors 
starting from $\lambda_1$ and $\phi_1$ using a certain minimal spanning tree. 
Our starting point is the
result established by A. Wentzell in 1970s \cite{wentzell1,wentzell2} (also see \cite{f-w}, Chapter 6). 

%%%
\subsection{Wentzell's formulas}
Wentzell's theorem  \cite{wentzell1,wentzell2} is valid for an arbitrary irreducible stochastic network with a finite number of states, not necessarily with detailed balance, where the pairwise
 transition rates are logarithmically equivalent to $\exp(-U_{ij}/T)$.
Being adapted for networks with detailed balance where the generator matrix is of the form \eqref{eq1}-\eqref{eq2}, Wentzell's theorem reads as follows.
\begin{theorem}
\label{the1}
Let $\lambda_1<\lambda_2<\ldots<\lambda_{n-1}$ be the positive eigenvalues of  $-L$ where $L$ is the generator matrix  given by Eq. \eqref{eq1}.
Let us define the numbers $V^{(k)}$ as 
\begin{equation}
\label{vk}
V^{(k)}=\min_{g\in G(k)}\sum_{(i\rightarrow j)\in g} \left(V_{ij}-V_{i}\right),
\end{equation}
where $G{(k)}$ is the set of $W$-graphs with the set  $W=W_k$ containing $k$ states.
Then for $T\rightarrow 0$ we have
\begin{equation}
\label{well}
\lambda_k\asymp e^{-(V^{(k)}-V^{(k+1)})/T},\quad k=1,2,\ldots,n-1.
\end{equation}
where the symbol $\asymp$ denotes the logarithmic equivalence.
\end{theorem}
We remind that a $W$-graph is defined as follows  \cite{f-w}.
\begin{definition}
\label{w-graph}
Let $S$ be the set of states. 
Let $W\subseteq S$ be its subset. The states in $W$ are called sinks. 
A $W$-graph is a directed graph defined on the set of states $S$ and possessing the following properties:
\begin{description}
\item[$(i)$]
Each state in $S\backslash W$ is the origin of exactly one arrow.
\item[$(ii)$]
There are no cycles in the graph.
\end{description}
Alternatively, $(ii)$ can be replaced with the condition that for every state  $i\in S\backslash W$ there 
exists a sequence of arrows leading from it to a sink $j\in W$.
\end{definition}

Thus, a $W$-graph with $k$ sinks can be constructed as follows. Pick $k$ sinks and
partition the rest of the states into $k$ subsets so that each of them contains exactly one sink. 
In each subset, draw arrows to connect the sets with the sink according to the rules in Definition \ref{w-graph}.
If states $i$ and $j$ are not connected by an edge we set $V_{ij}=\infty$.

Note that if $W=S$, the $W$-graph contains no edges. Hence $V^{(n)}$ in Eq. \eqref{vk} is zero.
Therefore, $\lambda_{n-1}\asymp V^{(n-1)}$, and the number $V^{(n-1)}$ is the smallest barrier in the network: 
$$V^{(n-1)} = \min_{i,j\in S,~i\sim j}(V_{ij}-V_i).$$
If the number of states in the system is small, one can calculate the numbers $V^{(k)}$, $k=1,2,\ldots,n-1$ 
directly using Eq. \eqref{vk} and find the asymptotics for the eigenvalues using Eq. \eqref{well}.
However, if the number of states is large, this approach becomes infeasible. 

In the next few Sections,  we will derive  recurrence relationships for
the numbers $V^{(k)}$ for the case where the pairwise rates are of the form of Eq. \eqref{eq1}, 
and dramatically simplify the calculation of
the asymptotic spectrum.

%%%%%%%%%%%%%%%%
%%%%%%%%%%%%%%%%
\subsection{The minimum spanning tree}
In this Section, we recall the definition of the minimum spanning tree and its crucial properties (see e.g. \cite{amo}). 
An undirected graph is called a tree if it consists of a single  connected component and contains no cycles.
Let $G(S,E,C)$ be a graph with the set of states $S$, the set of edges $E$, and the cost matrix $C=\{c_{ij}\}_{i,j\in S}$. 
If  states $i$ and $j$ are connected by an edge, the cost $c_{ij}$ is finite, otherwise $c_{ij}=\infty$.
\begin{definition}
\label{def:mst}
Let $G(S,E,C)$ be a connected graph. 
A spanning tree $\mathcal{T}=G(S,E',C)$ is a connected graph with the set of states $S$, the set 
of edges $E'\subset E$, and no cycles. The total cost of the spanning tree is defined as
$$\sigma(\mathcal{T}):=\sum_{(i,j)\in E'}c_{ij}.$$
A minimum spanning tree is a spanning tree whose total cost is minimal possible. 
\end{definition}

A minimum spanning tree has two important properties: it satisfies the cut optimality condition and the path optimality condition \cite{amo}.
A cut of a graph is a partition of its set of states into two subsets. The set of edges connecting states from the different subsets is called
a cut-set or also  a cut.  
The cut optimality condition states that a spanning tree $\mathcal{T}$ is a minimum spanning tree if and only if
for any edge $(i,j)\in\mathcal{T}$ $c_{ij}\le c_{kl}$ for every edge $(k,l)$ contained in the cut obtained by removing the edge $(i,j)$ from $\mathcal{T}$.  
The path optimality condition claims that a spanning tree $\mathcal{T}$ is a minimum spanning tree if and only if
for every edge $(k,l)\notin\mathcal{T}$, $c_{kl}\ge c_{ij}$ belonging to the unique path $w(k,l)\subset\mathcal{T}$ connecting the states $k$ and $l$. 

The cut optimality condition implies that the unique path $w^{\ast}(k,l)$ in a minimum spanning tree $\mst$ connecting the states $k$ and $l$
posesses the minimax property, i.e., 
\begin{equation}
\label{minimax}
\max_{(i,j)\in w^{\ast}(k,l)}c_{ij}=\min_{w(k,l)\in\mathcal{W}(k,l)}\max_{(i,j)\in w(k,l)}c_{ij},
\end{equation}
where $\mathcal{W}(k,l)$ is the set of all paths in $G(S,E,C)$ connecting $k$ and $l$.
We will call a path $w^{\ast}(a,b)$ connecting a pair of states $a$ and $b$ minimax if for any two
states $k,l\in w^{\ast}(a,b)$ the path $w^{\ast}(k,l)\subset w^{\ast}(a,b)$ satisfies Eq. \eqref{minimax}.
 
A minimum spanning tree does not need to be unique. 
If it is unique, then for each pair of states $k$ and $l$ there is a unique the minimax path.

For a network with pairwise rates given by Eq. \eqref{eq1} we define the cost $c_{ij}=V_{ij}$. 
This means that if the set of states of the network is equivalent to the set of  local minima of a potential energy landscape, 
and the edges correspond to the saddles separating the local minima, the cost
of the edge $(i,j)$ is the value of the potential at the saddle separating local minima $i$ and $j$.

For the rest of the paper, we will make the following genericness assumption.
\begin{assumption}
\label{as1}
The values of the potential at the states $V_i$, $i\in S$, and at the edges $V_{ij}$, $i,j\in S$, are all different.
Furthermore, all of the differences $V_{ij}-V_k$, $i,j,k\in S$, are also different.
\end{assumption}
In particular, this means that the minimum spanning tree 
where the cost $c_{ij}=V_{ij}$ is unique.
This minimum spanning tree $\mst$ is the key object for our construction.
The problem of finding  the minimum spanning tree is a well-studied (see e.g. \cite{amo}).
There exist a numbers of efficient algorithms for doing this. 

%\footnotetext[1]{
%The theorem proven in \cite{cam1} states that every Freidlin's cycle $C$ is exited via the edge corresponding to the 
%lowest saddle adjacent to it. This means that if one considers a cut
%$\{C, S\backslash C\}$, then this edge belongs to the minimum spanning tree by the cut optimality condition. 
%}

%%%%%%%%%%%%%%%%
%%%%%%%%%%%%%%%%
\subsection{Notations and Terminology}
In order to make our presentation clear and our equations compact, we introduce the following  notations.
\begin{itemize}
\item
A directed $W$-graph $g_{k}\in G(k)$ can be converted to a forest of $k$ trees by making all its edges undirected. We will denote this forest by $\mathcal{T}_k$.
(A graph that can be decomposed into a collection of trees is called a forest.)
\item
We will call  a $W$-graph  in $G(k)$, at which the minimum in Eq. \eqref{vk} is achieved, optimal, and  denote it by $g^{\ast}_k$. 
The corresponding  forest $\mathcal{T}_k^{\ast}$ will  also be called optimal.
\item
We will denote the $W$-set of the optimal graph $g_k^{\ast}\in G(k)$ by $W^{\ast}_k$, and call it the optimal set of sinks.
\end{itemize}

%%%%%%%%%%%%%
\subsection{Construction of asymptotic eigenvalues using the minimum spanning tree}
\label{sec:evals}
In this Section, we construct the set of numbers $\Delta_k$ determining the asymptotics for the eigenvalues 
using the minimum spanning tree. Simultaneously, we construct a collection of subsets $S_k\subset S$ whose
indicator functions give the asymptotics for the corresponding eigenvectors.
We start with the observation that Eq. \eqref{vk} defining the numbers $V^{(k)}$ can be rewritten as 
\begin{align}
V^{(k)}&=
\min_{g\in G(k)}\left(\sum_{(i, j)\in \mathcal{T}_k} V_{ij}-\sum_{i\in S\backslash W_k}V_{i}\right) =
\notag\\
&=\sum_{(i, j)\in \mst_k} V_{ij}+\sum_{i\in W^{\ast}_k}V_{i}-\sum_{i\in S}V_i,\label{vk2}
\end{align}
where $g^{\ast}_k\in G(k)$ is the optimal $W$-graph with $k$ sinks,  and $\mst_k$ and $W^{\ast}_k$ 
are the corresponding optimal forest and set of sinks. 
Therefore, the number $V^{(k)}$ is the sum of potentials $V_{ij}$ over the edges of the optimal forest plus the
sum of potentials over the optimal sinks minus the sum of potentials over all states. 
The last sum in Eq.  \eqref{vk2} is the same for all $W$-graphs $g_k$ and all $k=1,2,\ldots,n$.  
At this point, we can make the folowing  observation.
\begin{observation}
\label{o1}
Let $t$ be a connected component of the optimal $W$-graph $g^{\ast}_k$.
The sink $s\in t$ is the state with the minimal value of the potential among all states $i\in t$,
i.e., $V_s=\min_{i\in t}V_i$.
\end{observation}
If Observation \ref{o1} would not hold, we would be able to reduce the sum of 
potentials at the sinks while leaving optimal forest the same.

Unfortunately, the first two
sums in Eq. \eqref{vk2} cannot be optimized independently. 
If we sort the states and the edges of the minimum spanning tree 
in the ascending order according to their potentials and take the first $k$ states to be the sinks and the  first $n-k$ edges to constitute the forest, 
there is no guarantee that 
each subtree of the resulting forest contains exactly one sink. 
Therefore,  determination of the numbers $V^{(k)}$ is a nontrivial constrained optimization problem.
Below we propose a solution to it exploiting the nested property of the optimal $W$-graphs. We claim that
$(i)$ all optimal forests $\mst_k$ are subgraphs of the minimum spanning tree $\mst$,  and $(ii)$ 
the optimal $W$-graphs are nested. 
The former together with Eq. \eqref{vk2} immediately implies that 
\begin{equation}
\label{v1}
V^{(1)}=\sum_{(i,j)\in\mst}V_{ij}+\min_{i\in S}V_i-\sum_{i\in S}V_i.
\end{equation}
The latter means that all of the sinks of the optimal $W$-graph $g_k^{\ast}$ are also sinks of  $g^{\ast}_{k+1}$, and 
all of the edges of the optimal forest $\mst_{k+1}$ are also edges of $\mst_k$:
\begin{align}
W_k^{\ast}\subset W_{k+1}^{\ast},\quad k=1,2,\ldots, n-1,\label{w}\\
\mst_k\supset \mst_{k+1},\quad k=1,2,\ldots, n-1. \label{t}
\end{align}
Hence, in order to obtain the optimal $W$-graph $g^{\ast}_{k+1}$ from the optimal $W$-graph $g^{\ast}_k$, 
one needs to add exactly one sink and remove exactly one edge. Since each subtree of the optimal forest 
$\mst_{k+1}$ must contain exactly one sink, one needs to perform three optimal picks, the last two of which need to be done simultaneously:
\begin{itemize}
\item
pick a subtree $t$ of the optimal forest $\mst_k$, 
\item
split it into two subtrees by
removing one edge; denote the subtree containing the sink of $t$ by $t'$, and the other one by $t''$, and 
\item
pick a new sink in the subtree $t''$. 
\end{itemize}
Therefore, the numbers $V^{(k)}$ satisfy the following recurrence relationships:
\begin{align}
&V^{(k+1)}=V^{(k)}-\max_{t\in\mst_k}\max_{(p,q)\in t,i\in t''}(V_{pq}-V_i),\label{rr}\\
&{\rm where}\quad t=t'\cup t''\cup\{(p,q)\},~t''\cap W_k^{\ast}=\emptyset,\quad k=1,2,\ldots n-1.\notag
\end{align}
Assumption \ref{as1} guarantees that  in Eq. \eqref{rr},
the  optimal edge to remove and the optimal sink to add are unique.
We will denote them by
$(p^{\ast}_k,q^{\ast}_k)$ and $s^{\ast}_{k+1}$ respectively.
The asymptotics of the eigenvalue $\lambda_k$ is determined by the difference $V^{(k)}-V^{(k+1)}$ according to Theorem \ref{the1} \cite{wentzell1,wentzell2,f-w}.
Taking into account Eq. \eqref{vk2} we conclude that
\begin{equation}
\label{lambda}
\Delta_k:=V^{(k)}-V^{(k+1)} =V_{p^{\ast}_kq^{\ast}_k}-V_{s^{\ast}_{k+1}},\quad \lambda_k\asymp\exp(-\Delta_k/T).
\end{equation}
In the rest of this Section we will prove our claims stated above.

First we prove that all optimal forests $\mst_k$ are subgraphs of the minimum spanning tree $\mst$.
\begin{theorem}
\label{the2}
Suppose that Assumption \ref{as1} holds.
Then the optimal $W$-graphs $g_k^{\ast}\in G(k)$, $k=1,\ldots,n$
are subgraphs of the minimum spanning tree $\mst$.
\end{theorem}
\begin{figure}[htbp]
\begin{center}
\includegraphics[width=0.5\textwidth]{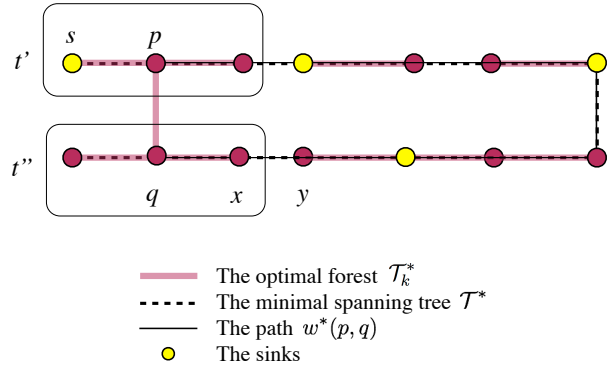}
\caption{Illustration for the proof of Theorem \ref{the2}.}
\label{fig:the2}
\end{center}
\end{figure}
\begin{proof}
We will proceed from converse.
Let $g^{\ast}_k\in G(k)$ be the optimal $W$-graph, and  $\mst_k$ be the corresponding optimal forest.
Suppose that $\mst_k$ contains an edge $(p,q)$ that does not belong to the minimum spanning tree $\mst$.
Suppose the edge $(p,q)$ belongs to a subtree $t$ of $\mst_k$.
Let $w^{\ast}(p,q)$ be the unique path in the minimum spanning tree $\mst$ connecting  the states $p$ and $q$. 
By the path optimality condition \cite{amo} combined  
with Assumption \ref{as1} we have
$$ V_{pq}>\max_{(i,j)\in w^{\ast}(p,q)} V_{ij}.$$
The removal of the edge $(p,q)$ splits the tree $t$ into two subtrees $t'$ and $t'$.
Without the loss of the generality we assume that $p\in t'$, $q\in t''$, and the sink $s$ of the tree $t$ belongs to $t'$.
Therefore, if we remove the edge $(p,q)$ from the forest $\mst_k$
and replace it with an edge $(x,y)\in w^{\ast}(p,q)$ such that $x\in t''$ and $y\notin t''$ as shown in Fig. \ref{fig:the2}, we transform the $W$-graph $g^{\ast}_k$
into another $W$-graph $g^{\star}_k$ with the same set of sinks and with a smaller sum of potentials over its edges.
This contradicts to the fact that $g^{\ast}_k$ is the optimal graph. Hence the optimal $W$-graph $g^{\ast}_k$
must contain only those edges that belong to the minimum spanning tree $\mst$.
\end{proof}
%%%%%%%%%%%%%%%%%%%

Now we prove the nested property of the optimal $W$-graphs and the recurrence relationship for the numbers $V^{(k)}$.
\begin{theorem}
\label{the3}
Suppose that Assumption \ref{as1} holds.
Then the optimal $W$-graphs are nested, i.e., Eqs. \eqref{w} and \eqref{t} hold, and the numbers $V^{(k)}$ satisfy the
recurrence relationships given by Eq. \eqref{v1} and \eqref{rr}.
\end{theorem}

The proof of Theorem \ref{the3} relies on
\begin{lemma}
\label{lemma1}
Suppose that Assumption \ref{as1} holds.
Then
\begin{description}
\item[$(i)$]
 the  sink $s^{\ast}_1$ of the optimal $W$-graph  $g^{\ast}_1$ is also a sink of the optimal $W$-graphs $g^{\ast}_k$, $k=2,3,\ldots,n$; 
 \item[$(ii)$]
 the edge $(p_1^{\ast},q_1^{\ast})$ that belongs to $\mst$ but does not belong to $\mst_2$, also does not belong to $\mst_k$, $k=3,\ldots,n$;
 \item[$(iii)$]
the second sink $s^{\ast}_2$ of the optimal $W$-graph $g^{\ast}_2$ is also a sink of $g^{\ast}_k$, $k=3,\ldots,n$.
\end{description}
\end{lemma}
Claim $(i)$ of Lemma \ref{lemma1} follows from Observation \ref{o1}.
Indeed, since the optimal graph $g^{\ast}_1$ is connected,
the state 
$$s_1^{\ast}=\arg\min_{i\in S}V_i$$ 
is the sink for all optimal $W$-graphs $g^{\ast}_k$, i.e., $s_1^{\ast}\in W_k$, $k=1,2,\ldots,n$.

The proof of Claim $(ii)$ is done from converse. The key point is to find an edge 
in the assumed-to-be-optimal $W$-graph $g^{\ast}_k$ to be replaced
with $(p_1^{\ast},q_1^{\ast})$ so that the sum in Eq. \eqref{vk2} decreases.
The choice of such an edge is different in different cases.
The proof of Claim $(iii)$ easily follows 
once Claim $(ii)$ is proven. The proofs of Claims $(ii)$ and $(iii)$ are found in the Appendix.

\begin{proof} (Theorem \ref{the3})
The optimal $W$-graph $g_1^{\ast}$ contains one connected component and one sink.
Eq. \eqref{v1} for $V^{(1)}$ immediately follows from Eq. \eqref{vk2} and Theorem \ref{the2}.

The optimal $W$-graph $g^{\ast}_2$ contains all edges of $g^{\ast}_1$ 
except for one than we denote by $(p_1^{\ast},q_1^{\ast})$, and two sinks, $s_1^{\ast}$ (by Lemma \ref{lemma1}, $(i)$)
 and $s_2^{\ast}$.
It follows from Eq. \eqref{vk2} that $(p_1^{\ast},q_1^{\ast})$ and $s_2^{\ast}$ satisfy
\begin{equation}
\label{pqs2}
\{(p_1^{\ast},q_1^{\ast}),s_2^{\ast}\}=\arg\max_{(p,q)\in\mst,~ i\in t''} (V_{pq}-V_i),
\end{equation}
where 
$
\mst=t'\cup t''\cup\{(p,q)\},~s_1^{\ast}\in t'.
$
Therefore,
 \begin{equation}
\label{v2}
V^{(2)}=\sum_{(i,j)\in\mst}V_{ij} - V_{p_1^{\ast}q_1^{\ast}}+V_{s_1^{\ast}}+ V_{s_2^{\ast}}- \sum_{i\in S}V_i = 
V^{(1)} -(V_{p_1^{\ast}q_1^{\ast}}- V_{s_2^{\ast}} ).
\end{equation}
Thus, Eqs. \eqref{v1}-\eqref{lambda} are valid for $k=1$.

By Lemma \ref{lemma1}, $(ii)$ and $(iii)$, 
the edge $(p_1^{\ast},q_1^{\ast})$ does not belong to $\mst_k$, $k=3,\ldots,n$, and
the sink $s^{\ast}_2$ of the optimal $W$-graph $g^{\ast}_2$ is also a sink of $g^{\ast}_k$, $k=3,\ldots,n$.
Therefore, we can restrict the further analysis to each of the connected components of the optimal $W$-graph $g^{\ast}_2$.
Applying Lemma \ref{lemma1} to each of the connected components we obtain that 
$(i)$ the sink $s^{\ast}_2$ of the optimal $W$-graph  $g^{\ast}_2$ is also a sink of $g^{\ast}_k$, $k=3,4\ldots,n$; 
$(ii)$ the edge $(p_2^{\ast},q_2^{\ast})$ that belongs to $\mst_2$ but does not belong to $\mst_3$, also does not belong to $\mst_k$, $k=4,\ldots,n$, 
and the third sink $s^{\ast}_3$ of $g^{\ast}_3$ is also a sink of $g^{\ast}_k$, $k=4,\ldots,n$.
Then we restrict the further analysis to each of the connected components of $g^{\ast}_3$.
Proceeding recursively, we prove the nested property of the optimal $W$-graphs given by Eqs. \eqref{w} and \eqref{t}.
Then the recurrence relationship for the numbers $V^{(k)}$ readily follows from the nested property and Eq. \eqref{vk2}.
\end{proof}

%%%%%%%
%%%%%%%

%%%%%%%%%%%%%%%%%%%%%%%%%%%%%%%%%
\subsection{Asymptotic eigenvectors, the optimal $W$-graphs, and Freidlin's cycles}
\label{sec:cycles}
In this Section, we discuss the relationship between the asymptotic eigenvectors, 
the optimal $W$-graphs, and Freidlin's cycles. Suppose that we have constructed the 
optimal $W$-graphs $g^{\ast}_1$, $g^{\ast}_2$, ..., $g^{\ast}_{k+1}$. 
Let  $s^{\ast}_{k+1}$ be the sink of $g^{\ast}_{k+1}$  that is not a sink of the optimal 
$W$-graphs $g^{\ast}_1$, $g^{\ast}_2$, ..., $g^{\ast}_{k}$.  Let us denote by $S_k$
the set of states in the connected component of the optimal forest  
$\mst_{k+1}$ containing the sink $s_{k+1}^{\ast}$. 
Then it follows from the theory developed 
in \cite{bovier2002} by Bovier and collaborators that the asymptotic eigenvector 
corresponding to the eigenvalue $\lambda_k\asymp V_{p_k^{\ast}q_k^{\ast}}-V_{s^{\ast}_{k+1}}$ is proportional to
the indicator function of the set $S_k$.  
I.e., if the temperature is sufficiently small, the eigenvector corresponding to $\lambda_k$ can be approximated by
\begin{equation}
\label{evec}
\phi_{k}=[\phi_{k}(1),\ldots,\phi_{k}(n)]^T,\quad{\rm where}\quad 
\phi_{k}(j)=\begin{cases}1,&j\in S_k,\\
0,&j\notin S_k.\end{cases}
\end{equation}
In addition to the set of states $S_k$ one also can consider the largest Freidlin's cycle $C_k\equiv C(s^{\ast}_{k+1})$ containing 
the sink $s^{\ast}_{k+1}$ and not containing any state with a smaller value of the potential. Below we will show that $C_k\subset S_k$. The significance of Freidlin's cycle $C_k$ is that if the system is originally in the set $S_k$, it 
will quickly get to $C_k$ and stay in $C_k$ prior to exiting  from the set $S_k$.
Hence the cycle $C_k$ can be  viewed as a metastable set of states of the network in the sense that if the system is originally in $C_k$ it will 
equilibrate in it prior to exiting it \cite{schuette03,schuette04}. It was proven in  \cite{bovier2002}, that the eigenvalue
$\lambda_{k}$ approaches the exit rate $r_k$ from the set $S_k$ which is equal to the exit rate from the cycle $C_k$ as 
the temperature tends to zero, i.e.,
$$
\lambda_k=r_k(1+o(1)).
$$

In the rest of this Section we will clarify the claim that the 
asymptotic eigenvector is the indicator function for the set $S_k$ 
and give an effective description of the Freidlin's cycle $C_k$. 
We will return to the discussion of metastability in Section \ref{sec:examples}.

If the temperature is small enough and Assumption 1 holds, then the eigenvalues satisfy
$$0<\lambda_1\ll\lambda_2\ll\ldots\ll\lambda_{k-1}\ll\lambda_k\ll\ldots.$$
Then the normalized eigenvector $\phi_k$ is approximately equal to  the normalized capacitor $h_{s_{k+1}^{\ast},W^{\ast}_k}$
(a.k.a. committor) \cite{bovier2002}, i.e.
\begin{equation}
\label{evec1}
\phi_k(j)  \approx  \frac{h_{s_{k+1}^{\ast},W^{\ast}_k}(j)}{\|h_{s_{k+1}^{\ast},W^{\ast}_k}\|},
\end{equation}
where the set $W^{\ast}_k$ is the optimal set of sinks in the $W$-graph $g^{\ast}_k$ and  
the capacitor $h_{s_{k+1}^{\ast},W^{\ast}_k}(j)$ is the probability that the process starting at state $j$ first reaches 
state $s^{\ast}_{k+1}$ rather then any state in the set $W^{\ast}_k$. The capacitor satisfies the backward Kolmogorov equation 
\begin{equation}
\label{cap}
\begin{cases}\sum_{i=1}^nL_{ij}h_{s_{k+1}^{\ast},W^{\ast}_k}(j) = 0,
& i\notin W^{\ast}_{k+1}=W^{\ast}_{k}\cup\{s^{\ast}_{k+1}\} ,\\
h_{s_{k+1}^{\ast},W^{\ast}_k}(i) = 0,& i\in W_{k}^{\ast},\\
h_{s_{k+1}^{\ast},W^{\ast}_k}(s^{\ast}_{k+1}) = 1.
\end{cases}
\end{equation} 
By our construction of the optimal $W$-graphs in Section \ref{sec:evals}, the highest potential barrier 
 separating any state $j\in S_k$
from state $s_{k+1}^{\ast}$ is smaller than the one separating it from any state in $W^{\ast}_k$, i.e., 
\begin{equation}
\label{bar}
\max_{(x,y)\in w^{\ast}(j,s^{\ast}_{k+1})} V_{xy}-
\min_{i\in  w^{\ast}(j,s^{\ast}_{k+1})}  V_i< 
V_{p^{\ast}_kq^{\ast}_k}-V_{s^{\ast}_{k+1}}\le
\max_{(x,y)\in w^{\ast}(j,s)}V_{xy}-\min_{i\in  w^{\ast}(j,s^{\ast}_{k+1}) } V_i
\end{equation}
for any $ j\in S_k$ and any $s\in W^{\ast}_k$ 
(here $w^{\ast}(a,b)$ is the unique path in the minimum spanning tree connecting states $a$ and $b$). 
Hence, as the temperature tends to zero, the process starting at state $j\in S_k$
will reach first $s^{\ast}_{k+1}$ rather than any state $s\in S\backslash S_k$ with probability tending to one.
On the other hand, by construction, 
for any state $j\in S\backslash S_k$, the highest barrier  separating it from the sink in the connected component 
of the optimal $W$-graph $g^{\ast}_{k+1}$ containing state $j$ is strictly less than $V_{p^{\ast}_kq^{\ast}_k}-V_{s^{\ast}_{k+1}}$.
Hence the probability to reach state $s^{\ast}_{k+1}$  rather than some sink in the set $W^{\ast}_k$ starting from state $j$
tends to zero as temperature tends to zero. 
Therefore, that the capacitor $h_{s_{k+1}^{\ast},W^{\ast}_k}$ approaches the indicator function of the set $S_k$.

Now we remind what are Freidlin's cycles. Originally, they were introduced by 
M. Freidlin in 1970s in order to describe the large time 
behavior of systems evolving according to the SDE $dx=b(x)dt+\sqrt{2T}dw$, where $x\in\mr^d$, $b(x)$ is
a continuously differentiable vector field, and $dw$ is the Brownian motion \cite{freidlin-cycles}. 
If  the parameter $T$ is small, the dynamics of this
system can be reduced to the dynamics of a continuous-time Markov chain where the states correspond to the
attractors of the system \cite{freidlin-cycles, freidlin-physicad,f-w}. 

Suppose that the vector field $b(x)$ is potential, i.e., $b(x)=-\nabla V(x)$, where $V(x)$ is
twice continuously differentiable and satisfies the following conditions:
(1) $V(x)$ is bounded from below, (2) $V(x)$ has  $n$ isolated local minima,
(3) all saddle points of $V(x)$ have different heights, and (4) $|V(x)|\rightarrow \infty$ as $|x|\rightarrow\infty$.
In this case,  the long time dynamics
of the system reduces to the continuous-time 
Markov chain with the generator of the form of Eq. \eqref{eq1}.
The hierarchy of Freidlin's cycles in this case was studied in \cite{cam1}. In particular, it was shown that the hierarchy of
cycles  is a full binary tree, whose leaves correspond to the potential minima or the states. They are called the zero order
cycles. In total, there are $2n-1$ cycles, 
and there is an isomorphism between the set of Freidlin's cycles and the set of edges of the minimum spanning tree. 
In \cite{freidlin-cycles,freidlin-physicad,f-w} the hierarchy of cycles was constructed using $W$-graphs. In \cite{cam1} the hierarchy of cycles
was constructed via a sequence of conversions of rate matrices into jump matrices and taking limits $T\rightarrow 0$.
Here we will give a simple and intuitive contruction. 
Its justification follows from \cite{cam1,freidlin-cycles,freidlin-physicad,f-w}.

Imagine the potential energy landscape $V(x)$, $x\in\mr^d$, and consider 
the sublevel sets
$$X_a:=\{x\in\mr^d~|~V(x)<a\},\quad a\in\mr.$$
The sets $X_a$ are compact.  For a fixed $a$, either the set $X_a$  is empty, or it consists of a finite number of connected components
each of which contains at least one local minimum. 
The collection of local minima belonging to the same connected component of $X_a$ forms a Freidlin's cycle.
Since all saddles are assumed to have different heights (Assumption \ref{as1}), 
each cycle consisting of more than one local minimum (i.e., of a nonzero order) can be decomposed into
is a union of exactly two subcycles. 
This shows that the hierarchy of cycles is a complete binary tree. 
Suppose we are gradually increasing the level number $a$ starting from 
$\min_{x\in\mr^d}V(x)$. There will be exactly $n-1$ saddles $x^{\ast}$ such that as $a$ reaches $V(x^{\ast})$, there occurs 
merging of two connected components of $X_a$ that used to be disjoing for some range of smaller values of $a$. 
These $n-1$ saddles  correspond to the edges of the minimum spanning tree.

Therefore, any Freidlin's cycle in the network with pairwise rates of the form of Eq. \eqref{eq1} can be defined as follows.
\begin{definition}
\label{def:fc}
A Freidlin's cycle $C$ containing a state $s^{\ast}\in S$ is a subset of states $C\subset S$ of the form 
 \begin{equation}
\label{cycle}
C=\left\{s\in S~\vline~\max_{(i,j)\in w^{\ast}(s^{\ast},s)}V_{ij}<a\right\},
\end{equation}
where $a$ is a constant and $w^{\ast}(s^{\ast},s)$ is the unique path in the minimum spanning tree 
connecting $s^{\ast}$ and $s$.
\end{definition}

The relationship between the optimal $W$-graphs and the Freidlin's cycles $C_k$ are given by 
\begin{theorem}
\label{the4}
Suppose that Assumption \ref{as1} holds.
Let $s^{\ast}_k$ be the sink of the optimal $W$-graph $g^{\ast}_k$ 
that is not a sink of any $g^{\ast}_j$, $j=1,2,\ldots,k-1$. 
Let $t_k$ be the subtree of the optimal forest $\mst_k$ containing the state $s_k^{\ast}$.
Then the largest Freidlin's cycle $C_k$ containing $s^{\ast}_k$ and not containing any state $s$ such that $V_s<V_{s^{\ast}_k}$ is
the subset of states of  $t_k$ 
satisfying
\begin{equation}
\label{ck}
C_k=\left\{s\in t_k ~\vline~\max_{(i,j)\in w^{\ast}(s_k^{\ast},s)}V_{ij}<V_{p^{\ast}_{k-1}q_{k-1}^{\ast}}\right\}.
\end{equation}
\end{theorem}
\begin{proof}
Let us consider the cut of the network partitioning the set of states $S$  as
$$S=\{i\in t_k\}\cup\{i\notin t_k\}.$$
Obviously, the edge $(p^{\ast}_{k-1},q^{\ast}_{k-1})$ belongs to the cut-set of this partition. 
We claim that the edge $(p^{\ast}_{k-1},q^{\ast}_{k-1})$ has the smallest value of the potential in this partition.
We proceed from converse. 
Suppose there is another edge $(p,q)$ in this cut-set such that $V_{pq}<V_{p^{\ast}_{k-1}q^{\ast}_{k-1}}$.
By the strong form of the cut optimality condition (see \cite{amo}, Section 13.3) $(p,q)$ belongs to the minimum spanning tree. Let us consider the $W$-graph
$g^{\star}_{k-1}$ that is obtained from $g^{\ast}_{k-1}$ by removing the edge $(p^{\ast}_{k-1},q^{\ast}_{k-1})$, adding the edge $(p,q)$,
and choosing the sinks properly.  Let $s^{\ast}_a$ and $s^{\ast}_b$ be the sinks of the connected components of the optimal $W$-graph 
$g^{\ast}_k$ adjacent to $t_k$ via the edges $(p^{\ast}_{k-1},q^{\ast}_{k-1})$ and $(p,q)$ respectively (see  Fig. \ref{fig:th4}).
\begin{figure}[htbp]
\begin{center}
\centerline{
\includegraphics[width=0.7\textwidth]{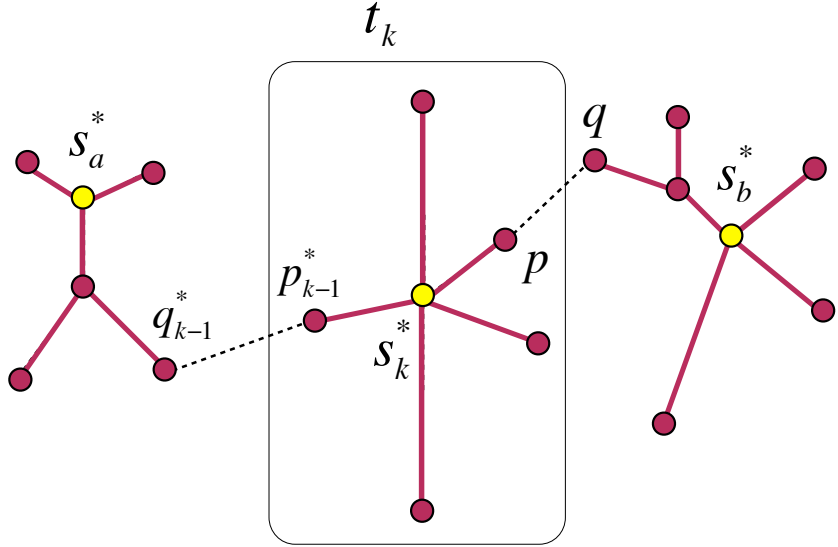}
}
\caption{Illustration for the proof of Theorem \ref{the4}.}
\label{fig:th4}
\end{center}
\end{figure}
Then the corresponding sinks of the $W$-graph $g^{\star}_{k-1}$ are $s_a^{\ast}$ and the one out of
$s^{\ast}_{k}$ and $s^{\ast}_b$ whose potential is smaller.
Since $V_{pq}<V_{p^{\ast}_{k-1}q^{\ast}_{k-1}}$ and $V_{s^{\ast}_b}\ge\min\{V_{s^{\ast}_{k}},V_{s^{\ast}_{b}}\}$,  
the sum in Eq. \eqref{vk2} for the $W$-graph $g^{\star}_{k-1}$ is smaller than the one for $g^{\ast}_{k-1}$. This contradicts to the
optimality of $g^{\ast}_{k-1}$. Therefore, the edge $(p^{\ast}_{k-1},q^{\ast}_{k-1})$ has the smallest value of the potential in the cut-set, i.e.,
$$
V_{p^{\ast}_{k-1}q^{\ast}_{k-1}}=\min_{p\in t_k,~q\notin t_k}V_{pq}.
$$
Therefore, the Freidlin's cycle containing $s^{\ast}_k$  and all other states $s$ such that 
$$
\max_{(i,j)\in w^{\ast}(s^{\ast}_k,s)}V_{ij}<V_{p^{\ast}_{k-1}q^{\ast}_{k-1}}
$$
belongs to $t_k$, i.e., it is the cycle $C_k$.

Next we observe that (see Fig. \ref{fig:th4})
$$
V_{p^{\ast}_{k-1}q^{\ast}_{k-1}}=\max_{(i,j)\in w^{\ast}(s^{\ast}_a,s^{\ast}_k)}V_{ij},
$$
and this maximum is unique by Assumption \ref{as1}. Hence, any larger Freidlin's cycle contains
 $s^{\ast}_a$ and $V_{s^{\ast}_a}<V_{s^{\ast}_{k}}$.
Therefore, the Freidlin's cycle $C_k$
is the largest cycle containing $s^{\ast}_k$ and not containing any state with a smaller value of the potential.
\end{proof}

%%%%%%%%%%%%%%%%%%%%%%%%%%%%%%%%%%
\section{An algorithm for computing the asymptotic spectrum}
\label{sec:algorithms}
In this Section we propose an algorithm to compute the asymptotics for the spectrum of the generator matrix 
$L$ starting from its low lying part. 
Central to the algorithm are the barrier function $u$ and the escape function $v$ defined as follows.
\begin{definition}
\label{def:u}
Let $W^{\ast}\subset S$ be a subset of states in the stochastic network with pairwise rates of the form \eqref{eq1}.
The barrier function $u(i)$ for the given set  $W^{\ast}$ is defined as 
\begin{equation}
\label{fbar}
u(i)=
\min_{s^{\ast}\in W^{\ast}}\max_{(p,q)\in w^{\ast}(i,s^{\ast})}V_{pq}
,\quad i\in S,
\end{equation}
where 
$w^{\ast}(s^{\ast},i)$ is the unique path in the minimum spanning tree connecting the states  $s^{\ast}$ and $i$.
\end{definition}
\begin{definition}
\label{def:v}
Let $W^{\ast}\subset S$ be a subset of states in the stochastic network with pairwise rates of the form \eqref{eq1}.
The escape function $v(i)$  for the given set of sinks $W^{\ast}$ is defined as 
$$v(i)=u(i)-V_i,\quad i\in S.$$
\end{definition}

 The output of the algorithm is the set of numbers 
$$
\Delta_k:=V_{p^{\ast}_kq^{\ast}_k}-V_{s^{\ast}_{k+1}}
$$
and the sets $S_k$ determiniming the asymptotics of the eigenvalues and the eigenvectors respectively, and
 the Freidlin's cycles $C_k$.
This Algorithm is justified by Theorems \ref{the2}, \ref{the3} and \ref{the4}.\newline
{\bf Algorithm 1:} Calculation of the asymptotic spectrum\newline
{\bf Initialization} \\
Precompute the minimum spanning tree $\mathcal{T}^{\ast}$.
Remove all edges that do not belong to  $\mathcal{T}^{\ast}$.
Set 
\begin{align*}
&k=0;\\
&s_1^{\ast}=\arg\min_{i\in S}V_i;  \\
&u(s_1^{\ast})=0,\quad u(i)= \max_{(p,q)\in w^{\ast}(s^{\ast}_{1},i)}V_{pq}, \quad i\in S;\\
&v(s_1^{\ast})=0,\quad v(i)=u(i)-V_i,\quad i\in S;\\
&\mst_1=\mst;\\
&S_0\equiv C_0 = S,
\end{align*}
where $w^{\ast}(s_{1}^{\ast},i)$ is the unique path in $\mst_k$ 
connecting the states $s_1^{\ast}$ and $i$. \newline
{\bf For $k=1:n-1$ } 
\begin{enumerate}
\item Find the new sink $s_{k+1}^{\ast}=\arg\max_{i\in S}v(i)$.
\item Find the cutting edge $(p_{k}^{\ast},q_{k}^{\ast})$ in the path 
in $\mst_k$  connecting the new sink $s_{k+1}^{\ast}$ with
one of the existing  sinks:
$$
w^{\ast}=\{s_j^{\ast},\ldots,p_k^{\ast},q_k^{\ast},\ldots,s_{k+1}^{\ast}\},\quad j\in\{1,2,\ldots,k\},
$$  
such that $u(p^{\ast}_k)<u(s_{k+1}^{\ast})$ and $u(q^{\ast}_k)=u(s_{k+1}^{\ast})$. Set
$$\Delta_{k}=(V_{p_{k}^{\ast}q_{k}^{\ast}}-V_{s^{\ast}_{k+1}}).$$
\item Remove the cutting edge $(p_{k}^{\ast},q_{k}^{\ast})$, i.e., set  $\mst_{k+1}=\mst_k\backslash \{(p_{k}^{\ast},q_{k}^{\ast})\}$.
\item Set $u(s_{k+1}^{\ast})=0$; $v(s^{\ast}_{k+1})=0$. 
\item Set  $S_{k}$ to be  the collection of states in the connected component of 
$\mst_k$ containing the sink $s_{k+1}^{\ast}$. For all states $i\in S_k$ 
update the barrier function $u$ and the escape function $v$:
$$
u(i)= \min\left\{u(i), \max_{(p,q)\in w^{\ast}(s^{\ast}_{k+1},i)}V_{pq}\right\},\quad v(i)=\min\{v(i),u(i)-V_i\},
$$ 
where $w^{\ast}(s_{k+1}^{\ast},i)$ is the unique path in $\mst_k$ 
connecting the states $s_k^{\ast}$ and $i$. 
The sink $s^{\ast}_{k+1}$ and the set of states where the values of $u$ and $v$ have changed constitute the Freidlin's cycle $C_{k}$.

\end{enumerate}
{\bf end for} 

There exists a collection of greedy algorithms for finding the minimum spanning tree \cite{amo}.
We have used Kruskal's algorithm  \cite{kruskal,amo} whose  computational cost for a network  with $n$ states and $m$ edges 
is $O(m+n\log n)$ plus the time of sorting the edges \cite{amo}.

The initialization and Step 5 in the for-cycle is done using a recursive procedure in at most $n-k$ steps 
because the minimum spanning tree and its subgraphs contain no cycles.
Step 1 in the for-cycle is done using the heap sort whose cost is $\log (n-k)$.
Step 2, finding the cutting edge, requires at worst $l$ steps if the path $w^{\ast}(s^{\ast}_{k+1},s^{\ast}_k)$ 
consists of $l$ edges. Obviously, $l\le n-k$, and  typically $l\ll n-k$.

Therefore, the upper bound for the computational cost of the for-cycle is $O(n(n-1)+n\log n-n)=O(n^2- 2n +n\log n)$.

We remark that one can replace the for-cycle with the while-cycle in Algorithm 1 with the stopping criterion of the form
$V_{p^{\ast}_kq^{\ast}_k}-V_{s^{\ast}_{k+1}}<\Delta$. 

%%%%%%
\begin{figure}[htbp]
\begin{center}
\centerline{
\includegraphics[width=\textwidth]{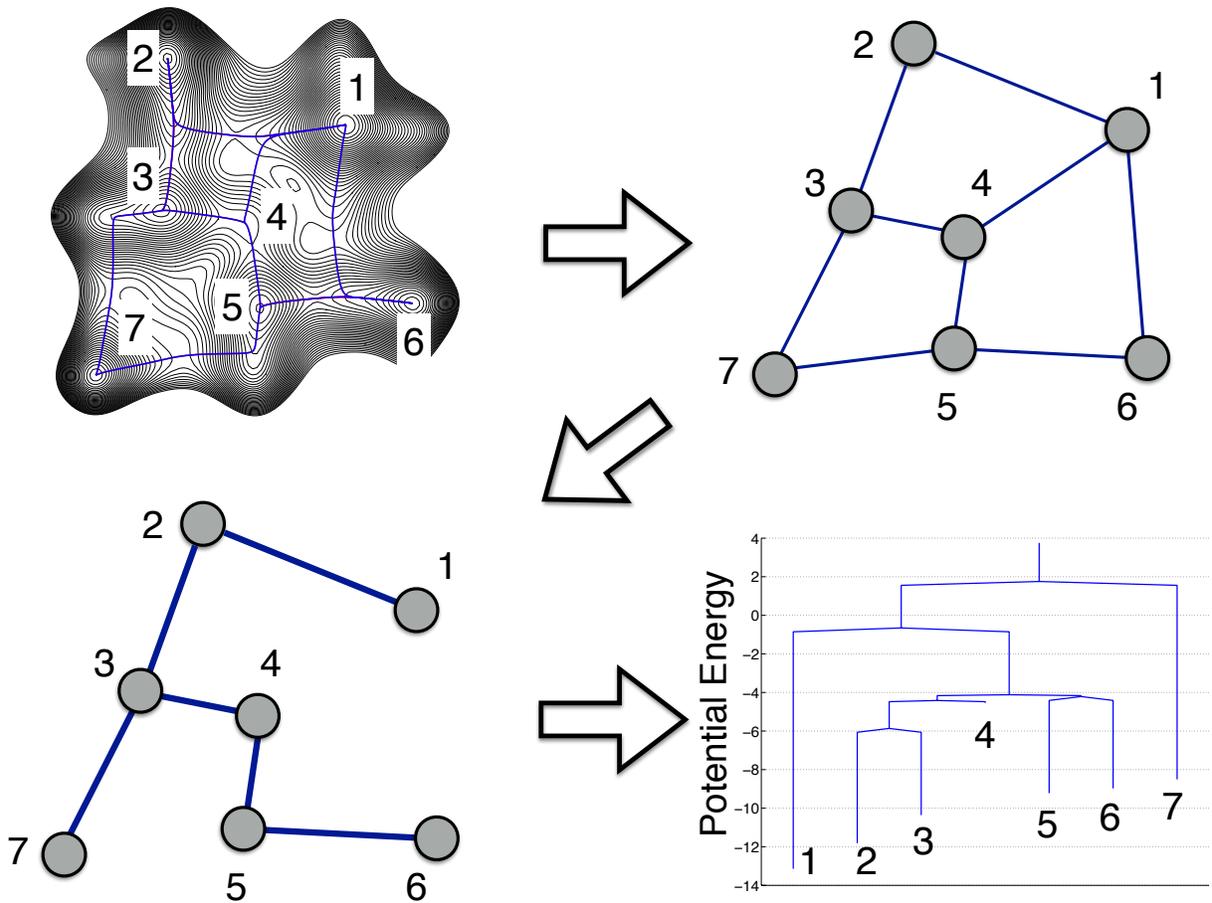}
}
\caption{Example: the seven-well potential. The potential energy landscape is converted into a stochastic network. 
Then the minimum spanning tree and 
the disconnectivity graph are built.}
\label{fig:7well1}
\end{center}
\end{figure}
We demonstrate how Algorithm 1 works on the example of the seven-well potential (Figures \ref{fig:7well1} and \ref{fig:7well2}).
The continuous potential energy landscape (Figure \ref{fig:7well1}, top left) is converted into a 
stochastic network with 7 states corresponding to the potential minima (Figure \ref{fig:7well1}, top right). 
A pair of states is connected by an edge if and
only if there exists a Minimum Energy Path (MEP) connecting them that does not pass through other minima. 
The resulting network 
contains 9 edges. The numbers $V_i$, $i=1,\ldots,7$, are the values of the potential at the corresponding minima.
The numbers $V_{ij}$, $i,j\in\{1,\ldots,7\}$, $i\neq j$, are the maximal values of the potential along 
the corresponding MEPs, i.e., 
the values of the potential at the corresponding saddles. Then we extract the minimum spanning tree 
(Figure \ref{fig:7well1}, bottom left) that can be easily converted 
into the disconnectivity graph (Figure \ref{fig:7well1}, bottom right). 

Since state 1 corresponds to the deepest minimum, we set $s^{\ast}_1=1$. 
The saddle separating minima 1 and 2 is higher than those separating 
minima 2, 3, 4, 5, and 6, but lower than the one separating all of them from minimum 7. 
The value function $u$ and the escape function $v$  are initialized as shown in Figure \ref{fig:7well2}, top left.  The set $S_0$ as well as 
Freidlin's cycle $C_0$ are always the whole set of states. Then the for-cycle at $k=1$ gives the following.
The maximum of $v$ is reached at state 2.  Hence state 2 becomes the new sink $s^{\ast}_2$. The cutting edge $(p_1^{\ast},q_1^{\ast})$ is the edge $(1,2)$.
We remove it from the network. Hence the set $S_1$ is $\{2,3,4,5,6,7\}$.
We update the functions $u$ and $v$ starting the computation from state 2. 
State 1 does not belong to the same connected component as the new sink 2,
therefore, $u(1)$ and $v(1)$ are not updated. 
State 7 belongs to the same connected component as state 2. 
However, since the highest barrier separating states 1 and 7 is the same as the one
separating states 2 and 7, the values $u(7)$ and $v(7)$ remain the same. At the rest of the states, both values $u(i)$ and $v(i)$ are updated.  
Hence the Freidlin's cycle is $C_1=\{2,3,4,5,6\}$ (Figure \ref{fig:7well2}, top middle).
\begin{figure}[htbp]
\begin{center}
\centerline{
\includegraphics[width=\textwidth]{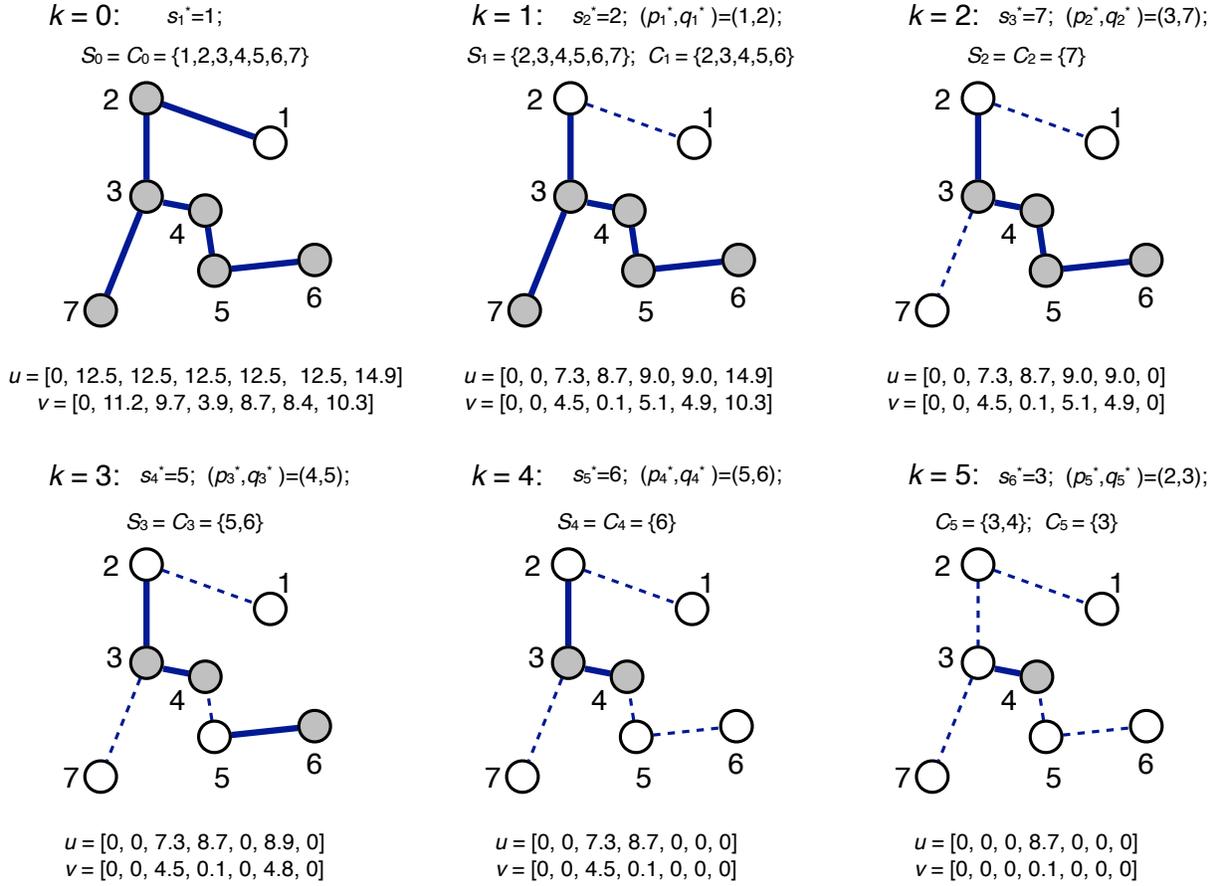}
}
\caption{Example: the application of Algorithm 1 to the stochastic network in Figure \ref{fig:7well1}. 
 The functions $u$ and $v$ are computed and 
then updates at every step. The sequences of 
the sinks $s^{\ast}_j$, the cutting edges $(p_j^{\ast},q_j^{\ast})$, and the corresponding Freidlin's cycles $C_{j}$ are built in the process.
}
\label{fig:7well2}
\end{center}
\end{figure}
Continuing in a similar manner for $k=2,3,\ldots$, we obtain the following sequences of sinks, cutting edges, sets $S_k$
and the corresponding Freidlin's cycles:
\begin{align*}
&s^{\ast}_1=1,\quad S_0 = C_0=\{1,2,3,4,5,6,7\},\\
&s^{\ast}_2=2,\quad(p_1^{\ast},q_1^{\ast})=(1,2),\quad S_1 = \{2,3,4,5,6,7\},\quad C_1=\{2,3,4,5,6\},\\
&s^{\ast}_3=7,\quad(p_2^{\ast},q_2^{\ast})=(3,7),\quad S_2 = C_2=\{7\},\\
&s^{\ast}_4=5,\quad(p_3^{\ast},q_3^{\ast})=(4,5),\quad S_3= C_3=\{5,6\},\\
&s^{\ast}_5=6,\quad(p_4^{\ast},q_4^{\ast})=(5,6),\quad S_4 = C_4=\{6\},\\
&s^{\ast}_6=3,\quad(p_5^{\ast},q_5^{\ast})=(2,3),\quad S_5 = \{3,4\},\quad C_5=\{3\},\\
&s^{\ast}_7=4,\quad(p_6^{\ast},q_6^{\ast})=(3,4),\quad S_6 = C_6=\{4\}.
\end{align*}
These sequences define the asymptotic eigenvalues and eigenvectors:
\begin{align*}
&\lambda_0=0,\quad \phi_0=[1,1,1,1,1,1,1]^T,\\
&\lambda_1\asymp \exp(-(V_{12}-V_2)/T),\quad \phi_1=[0,1,1,1,1,1,1]^T,\\
&\lambda_2\asymp \exp(-(V_{37}-V_7)/T),\quad \phi_2=[0,0,0,0,0,0,1]^T,\\
&\lambda_3\asymp \exp(-(V_{45}-V_5)/T),\quad \phi_3=[0,0,0,0,1,1,0]^T,\\
&\lambda_4\asymp \exp(-(V_{56}-V_6)/T),\quad \phi_4=[0,0,0,0,0,1,0]^T,\\
&\lambda_5\asymp \exp(-(V_{23}-V_3)/T),\quad \phi_5=[0,0,1,1,0,0,0]^T,\\
&\lambda_6\asymp \exp(-(V_{34}-V_4)/T),\quad \phi_6=[0,0,0,1,0,0,0]^T.
\end{align*}

%%%%%%%%%%%%%%%%%%%%%%%%%%%
%%%%%%
%%%%%%     LJ38
%%%%%%%
%%%%%%%%%%

\section{Application to  the Lennard-Jones-38 network}
\label{sec:examples}
The potential energy of a Lennard-Jones cluster LJ$_N$ is given by
\begin{equation}
\label{LJpe}
V(\mathbf{r}) = 4\epsilon\sum_{i<j}\left[\left(\frac{\sigma}{r_{ij} }\right)^{12} - \left(\frac{\sigma}{r_{ij} }\right)^{6}\right],
\end{equation}
where the numbers $r_{ij}=|\mathbf{r}_i-\mathbf{r}_j|$ are the pairwize distances between the atoms.
Throughout this work we will  use reduced units with $k_B=\epsilon=\sigma=1$.
The majority of global potential energy minima for Lennard-Jones clusters of various sizes are based on the icosahedral packing.
However, for some special numbers of atoms, Lennard-Jones clusters may admit a high symmetry configuration based on other packings 
\cite{wales-doye,wales38, wales_landscapes}.
The smallest  special number is 38. 
The  potential energy minimum of the $\lj38$ cluster is achieved at
 the face-centered cubic truncated octahedron with the point group $O_h$ (Fig. \ref{fcc_ico}).
The second lowest minimum is the icosahedral structure with the $C_{5v}$ point group  (Fig. \ref{fcc_ico}). 
For brevity we will refer to these configurations as FCC and ICO respectively. 
These two lowest minima are far disconnected in the configurational space. 
It was shown by Frank in 1950s \cite{frank} that as a monoatomic liquid cools, structures based on the icosahedral packing tend to appear.
However, in order to crystalize, the atoms should rearrange into a periodically-extendable structure, e.g., face-centered cubic.
\begin{figure}[htbp]
\begin{center}
\centerline{
\includegraphics[width=\textwidth]{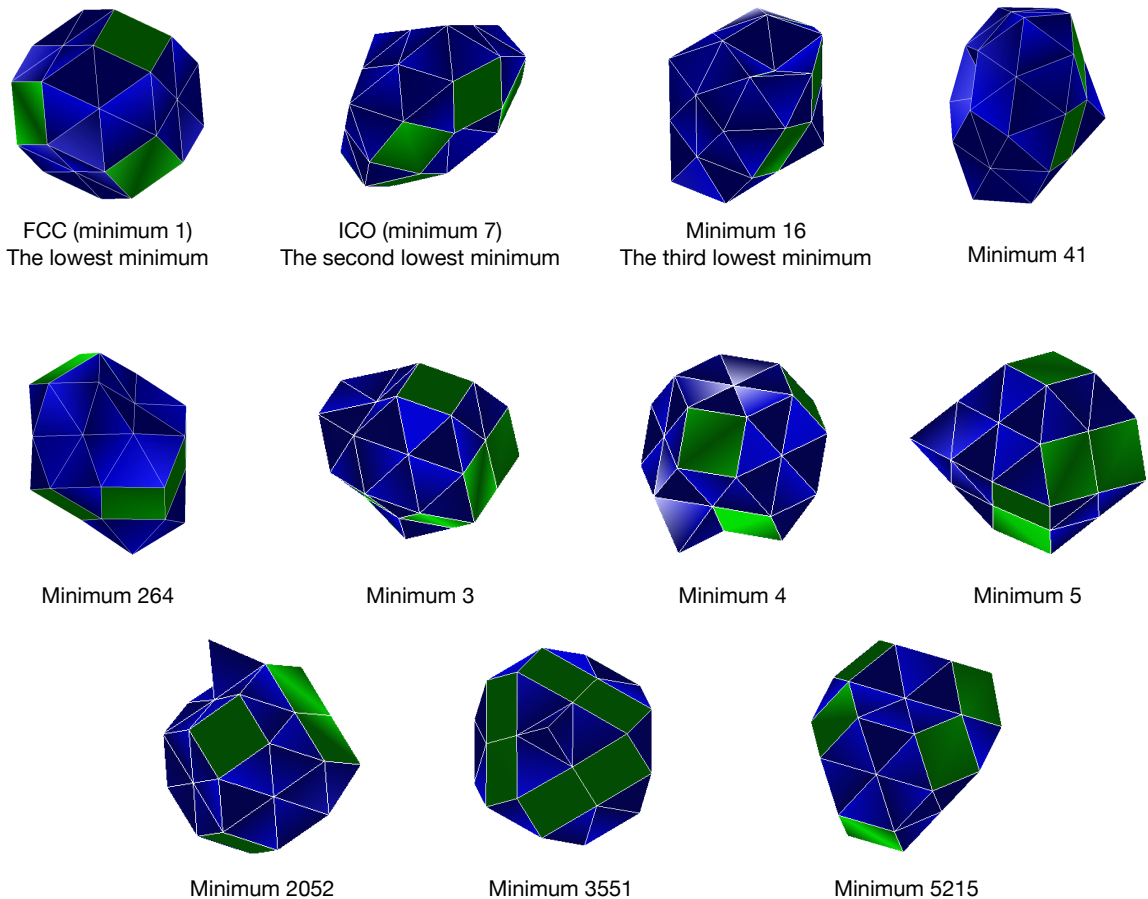}
}
\caption{Some important local minima of the potential energy of the LJ$_{38}$. 
}
\label{fcc_ico}
\end{center}
\end{figure}

Wales and collaborators developed an efficient technique for  conversion of potential energy landscapes
into stochastic networks whose states and edges correspond to local minima  
and transition states (saddles of Morse index one separating pairs of local minima) respectively \cite{wales38,wales_landscapes,wales_book}. 
The stochastic network associated with  $\lj38$ is publicly available via Wales's group web site \cite{wales_network}.
Its connected component containing 
FCC and ICO (minima 1 and 7 in Wales's list respectively) contains 71887 states and 119853 edges.
We will denote the states in the $\lj38$ network other than FCC and ICO by their index in Wales's list.

The problem of the $\lj38$ cluster rearrangement has attracted a lot of attention in the past fifteen years 
and has become a benchmark problem in chemical physics.
Many scientists attacked the problem of $\lj38$ rearrangement between its two lowest 
potential minima FCC and ICO using different tools. Wales analyzed the $\lj38$ network using the
Discrete Path Sampling \cite{wales0,wales1,wales_landscapes}.
The asymptotic zero-temperature path connecting FCC and ICO and the sub-hierarchy of Freidlin's cycles involved into the 
transition process was found in \cite{cam1}. 
A finite temperature analysis of the $\lj38$ network using the tools of the Transition Path Theory
was recently conducted in \cite{cve}.
The $\lj38$ cluster rearrangement in the continuous setting was also attacked by methods
that do not involved the exhaustive study of the energy landscape. These methods include
direct transition current sampling \cite{picciani}, 
molecular dynamics and temperature accelerated molecular dynamics \cite{voter}, and
parallel tempering \cite{neirotti}.  

The barrier separating FCC and ICO has the height of 4.219 and 3.543 energy units with respect to FCC and ICO respectively \cite{wales38}.
Typically, $\lj38$ is considered at low temperatures $0<T\ll 1$ as the solid-solid phase transition between face-centered cubic and
icosahedral structures takes place at $T=0.12$, the outer layer starts to melt at $T=0.18$, 
and the cluster melts completely at $T=0.35$ \cite{frantsuzov}. 
The barrier, separating ICO from FCC is about 30 $k_BT$ at $T=0.12$.
One might expect that the icosahedral basin with the deepest minimum ICO is, in some sense, 
a metastable subset of the $\lj38$ network. 
Our results show, however, that the situation is delicate. Whether to view the icosahedral basin as metastable or not
depends upon what definition of metastability is used and the observation time as well.

The graph of 
$
\Delta_k:=V_{p^{\ast}_kq^{\ast}_k}-V_{s^{\ast}_{k+1}}
$
versus $k$ for $k=1,\ldots,71886$ is shown in Fig. \ref{fig:delta}. Recall that $\lambda_k \asymp\exp(-\Delta_k/T)$. More or less notable 
gaps are present only between the first few barriers $\Delta_k$ corresponding to sinks with high potential evergy. These sinks are separated from
the rest of the states by very high potential barriers.
The eigenvalue corresponding to the sink ICO is $\lambda_{245}$.
There is no significant gap separating  $\Delta_{245}$:
 $\Delta_{246}-\Delta_{245}\approx0.0036$. This means that $\lambda_{245}\ll\lambda_{246}$ only for extremely low temperatures (at least, $T$ should be less than
0.0036).
\begin{figure}[htbp]
\begin{center}
\centerline{(a)
\includegraphics[width=\textwidth]{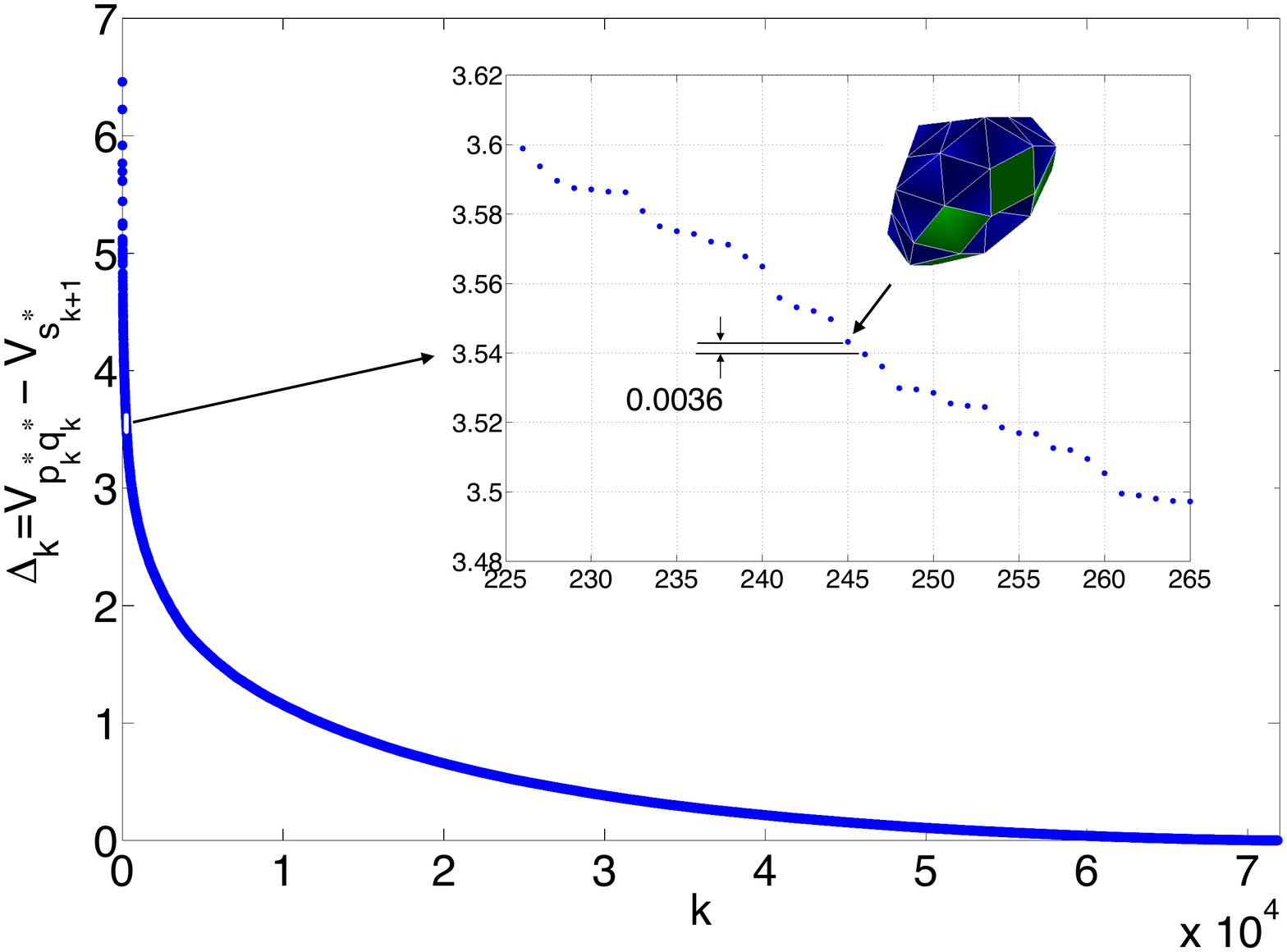}
}
\caption{
 The numbers $\Delta_k:=V_{p^{\ast}_{k-1}q^{\ast}_{k-1}}-V_{s^{\ast}_{k}}$ 
versus $k$ for the $\lj38$ network.
}
\label{fig:delta}
\end{center}
\end{figure}
The disconnectivity graph for the sinks from $s^{\ast}_1\equiv$FCC up to $s^{\ast}_{300}$ is shown in Fig. \ref{fig:dgraph300}.
ICO is the sink $s^{\ast}_{246}$. This graph shows that if the system is initially at ICO or FCC, it is extremely
unlikely for it to get to any other sink out of the first 300,
if the temperature $T<0.1$. 
Therefore, the sinks corresponding to the smallest eigenvalues are essentially irrelevant to the low-temperature dynamics.
This means that if the system is initially not in one of these states, and the observation time is not extremely long, 
it is unlikely for the system to reach those states. A relevant discussion can be found in \cite{wales2014}.

\begin{figure}[htbp]
\begin{center}
\centerline{
\includegraphics[width=\textwidth]{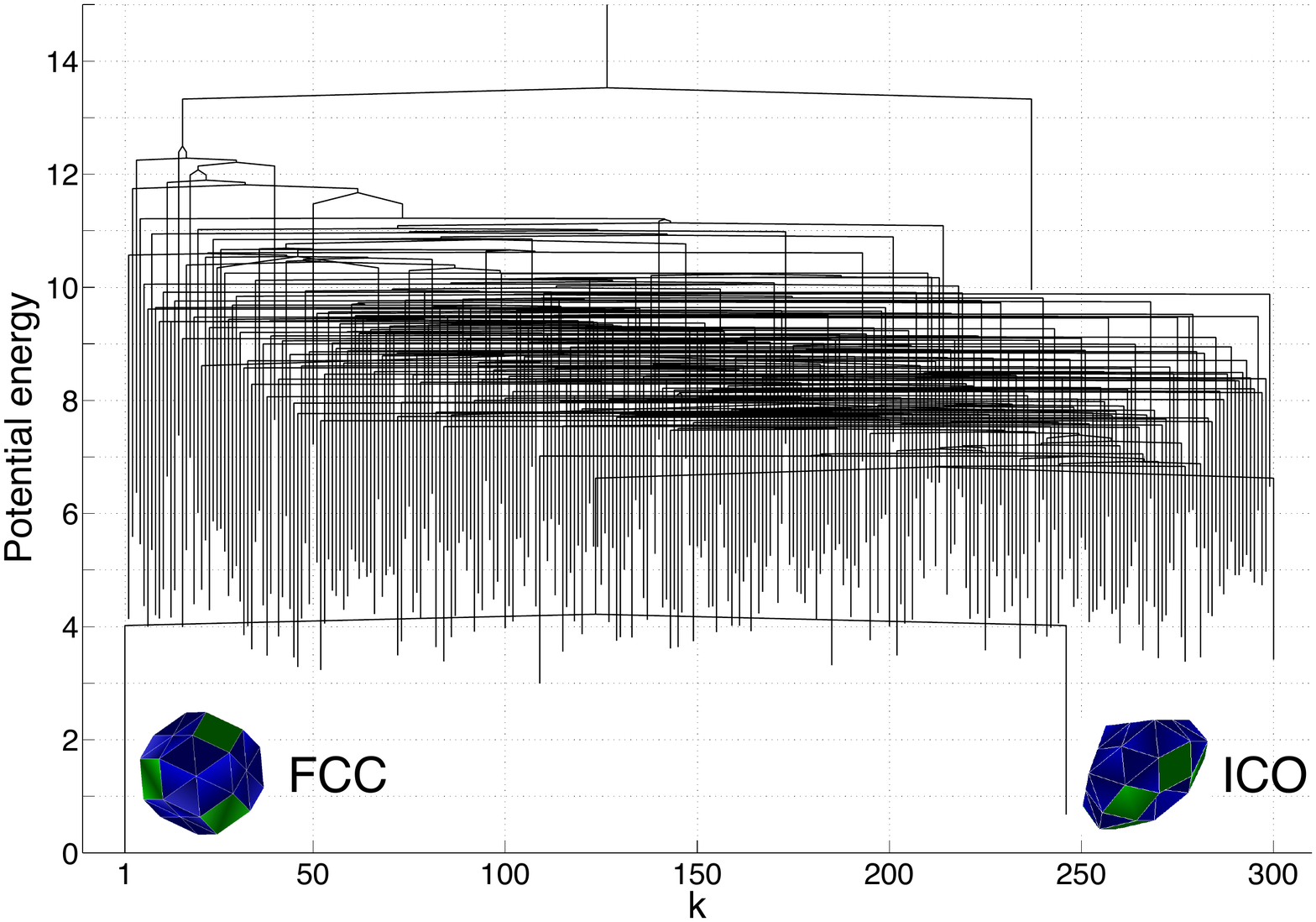}
}
\caption{
The disconnectivity graph showing the first 300 sinks of the $\lj38$ network.
FCC and ICO correspond to sinks 1 and 246 respectively. The states are ordered according to the
number of sink that they represent. The potential energy is shown relative to FCC.
}
\label{fig:dgraph300}
\end{center}
\end{figure}

Algorithm 1 also gives the collection of sets $S_k$ determining the asymptotic eigenvectors, and the corresponding Freidlin's cycles $C_k$.
A few largest disjoint sets $S_k$, $k\ge1$, are shown in Fig. \ref{fig:scheme}. The largest set $S_k$ for $k\ge 1$ is  $S_{245}$, the one which appears when the sink
corresponding to the second lowest minimum ICO is added.  It consists of 56290 states. Freidlin's cycle $C({\rm ICO})$ contains 791 states. This means that if the 
temperature is low enough and the system is initially at any state belonging to $S_{245}$, it relatively quickly gets to $C({\rm ICO})\subset S_{245}$ and stays there for relatively long time
$O(\exp(-\Delta_{245}/T))$ prior to exiting it. The other large disjoint sets $S_k$, $k\ge 1$, are $S_{6910}$ with 4252 states, the corresponding sink is minimum 3, and the corresponding Freidlin's
cycle contains 3 states; $S_{7482}$ with 1316 states, corresponding to minimum 4, and $|C(4)|=1$; $S_{5296}$ with 379 states,  corresponding to minimum 5, and $|C(5)|=2$; 
$S_{4143}$ with 990 states, corresponding to minimum 5215, and 
$|C(5215)| = 8$; $S_{11750}$ with 680 states,  corresponding to minimum 3551,  and $|C(3551)|=7$; and $S_{11961}$ with 1758 states,  corresponding to minimum 2052,  and $|C(2052)|=4$.
The relationship between these sets outlined in Fig. \ref{fig:scheme} is obtained using the algorithm for computing the asymptotic zero-temperature path introduced in \cite{cam1}. 
Besides the states belonging to one of the shown sets $S_k$,  there are 6221 more states (excluding FCC) in the $\lj38$ network that do not belong to any of the shown sets. 
The largest set $S_k$, $k\ge 1$ formed by  these remaining 6221 states is $S_{8009}$ with 288 states corresponding to minimum 587, and $C(587)$ consists of 2 states. 
The next largest disjoint sets $S_k$, $k\ge 1$, formed by the remaining states consist of 160,  98,   87,  79,  $\ldots$ states.  
Overall, the set of states in the $\lj38$ network can be decomposed into a disjoint union of the global minimum FCC and 2327 sets $S_k$. 
 Out of them, 1395 sets consist of a single state, 
406 consist of 2 states, 177 consist of 3 states, etc. The complete data about these disjoint sets $S_k$ are found in Table 1.
\begin{figure}[htbp]
\begin{center}
\centerline{
\includegraphics[width=\textwidth]{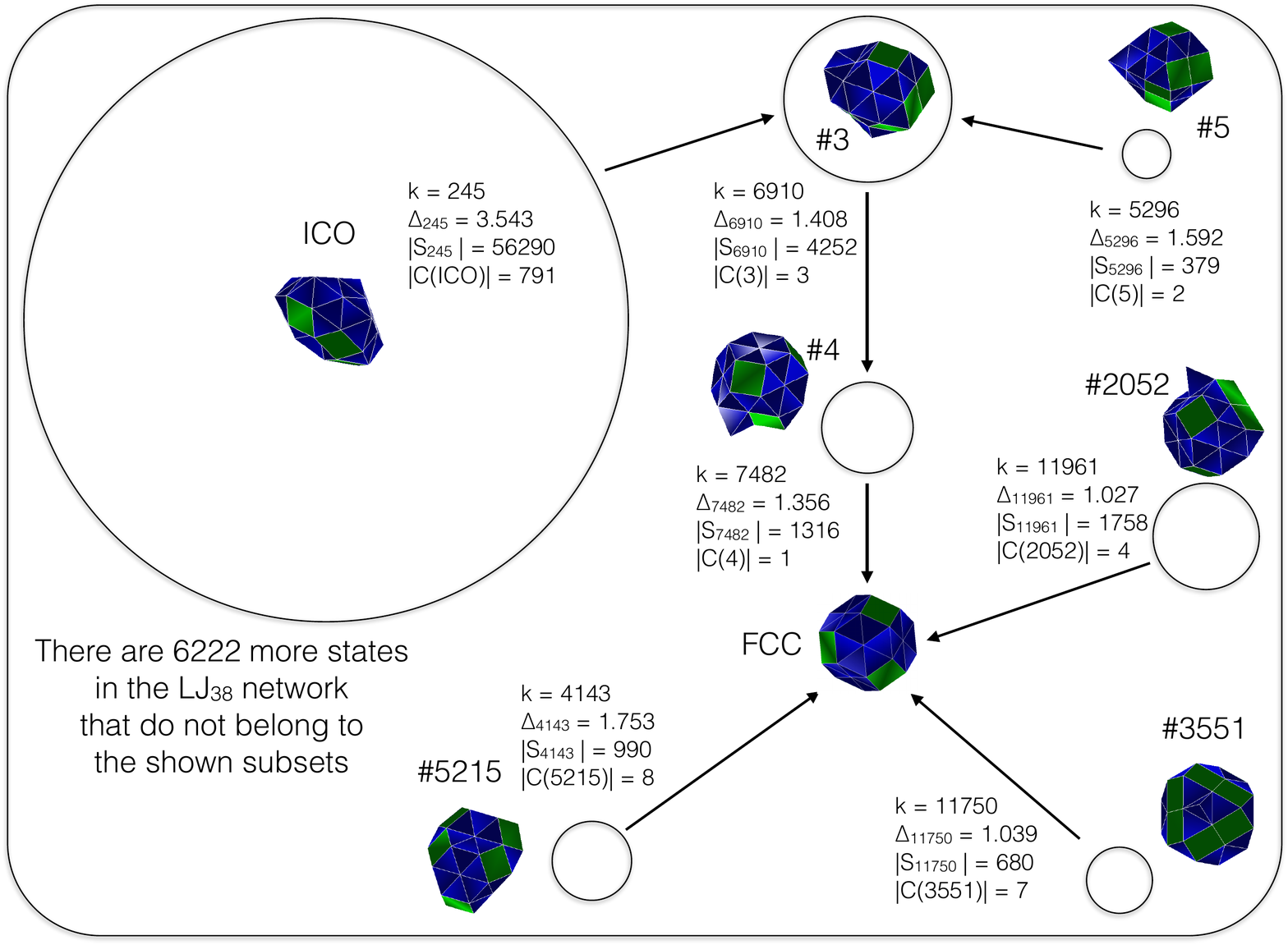}
}
\caption{
The largest disjoint sets $S_k$'s in the $\lj38$ network together with the  corresponding $\Delta_k$'s and Freidlin's cycles.
The area of the circles representing the sets $S_k$ is proportional to the number of states in them.
}
\label{fig:scheme}
\end{center}
\end{figure}

    \begin{table}
\label{table1}
\begin{center}
 \scalebox{0.7}{
   \begin{tabular}{ |c |c |c |}
    \hline \\
    $N$ & The \# of sets $S_k$ with $|S_k|=N$ & Sink(s) \\  \hline
     56290 & 1 & ICO\\
     4252 & 1 & 3\\
     1758 & 1 & 2052\\
     1316 & 1 & 4\\
     990 & 1 & 5215\\
     680 & 1 & 3551\\
    379 & 1 & 5 \\
     288 & 1 & 587\\
    160 & 1 & 5429\\
     98 & 1 & 2295\\
     87 & 1 &  9087\\
    79 & 1 & 4305\\
    66 & 1 & 3552\\
    54 & 1 & 7746\\
     49 & 1 & 30562\\
    47 & 1 &  13165\\
    45 & 1 & 407 \\
    40 & 1  & 17251\\
    36 & 1 & 3074\\
  33 & 1 & 4065\\
     28 &3  & 45155,  77289, 85766\\
    27 & 1 & 3191\\
    25 & 3 & 3863, 32036, 75247\\
    24 & 1& 85341 \\
     23 & 1 & 6757 \\
   21 & 2 & 2, 6070\\
   20 & 1 & 4066\\
   18 & 1 & 11218 \\
   17 & 3 & 18648, 36425, 39076\\
   16 & 4 & 16545, 24258, 33579, 79028\\
   15 & 5 & 11238, 29369, 59335, 70722, 94195\\
   14 & 1 & 9833\\
   13 &  3 & 13287, 35221, 51978\\
   12 & 10 &\\
   11 & 11 & \\
   10 & 14 &\\
   9 & 18&\\
   8 & 21 &\\
   7 & 27 &\\
   6 & 47 &\\
   5 & 55 &\\
   4 & 97 &\\
   3 & 177 &\\
   2 & 406 &\\
   1 & 1395 &\\
\hline
    \end{tabular}
    }
\caption{The sizes of disjoint sets $S_k$ constituting the set of states of the $\lj38$ network together with FCC.
The indicator functions of the sets $S_k$ are asymptotic eigenvectors.}
\end{center}
\end{table}

Fig. \ref{fig:dgraph300} suggests that some of the sets $S_k$ are separated by high potential barriers from the global potential minimum FCC.
This fact motivates us to restrict our attention to the part of the $\lj38$ network that is accessible from FCC at low temperatures if the 
observation time is large but not very large. We take the decomposition of the $\lj38$ network into the disjoint union of FCC and 2327 sets $S_k$ and select only those $S_k$'s
that are separated from FCC by a barrier whose height does not exceed 5 relative to $V_{\rm FCC}$  (i.e., for these $S_k$'s,  $V_{p^{\ast}_kq^{\ast}_k}-V_{\rm FCC}<5$
or $V_{p^{\ast}_kq^{\ast}_k}<-168.928$). All 60 such sets $S_k$, $k\ge 1$, are listed in Table 2.  
Table 2 shows that there is a significant spectral gap for the truncated and factored $\lj38$ network:  $\Delta_{245}-\Delta_{4143}=1.790$.
The truncated and factored minimum spanning tree for the $\lj38$ network formed by these selected sets and FCC in Fig. \ref{fig:red}
is calculated using the algorithm introduced in \cite{cam1}.
Lumping the states into disjoint sets $S_k$ can be helpful for comparison with electron microscopy or 
diffraction experiments since large collection of states  \cite{wales_network} 
based on icosahedral packing (395 states, states 6 through 400) is indistinguishable from low resolution experimental data. 
Similarly, states 1 -- 5 \cite{wales_network}
based on face-centered cubic packing are also indistinguishable.  Therefore, for a careful comparison, even further lumping may be done.
We leave this problem for the future. 
\begin{figure}[htbp]
\begin{center}
\centerline{
\includegraphics[width=\textwidth]{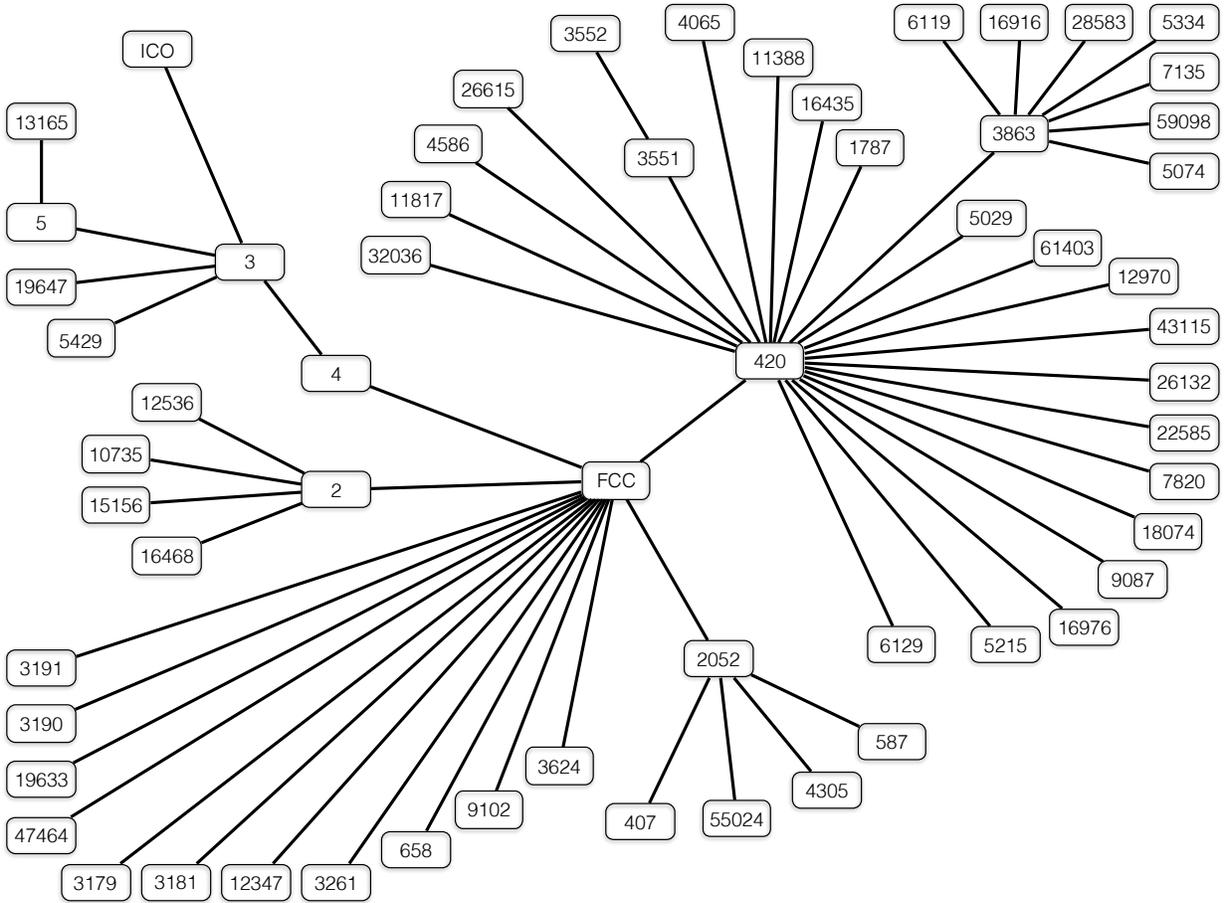}
}
\caption{
The truncated and factored minimum spanning tree for the $\lj38$ network.
}
\label{fig:red}
\end{center}
\end{figure}

    \begin{table}
\label{table2}
\begin{center}
\scalebox{0.7}{
    \begin{tabular}{ |c |c |c | c|c|c|}
    \hline \\
    $k$ & Sink & $V_{p^{\ast}_kq^{\ast}_k}$ & $\Delta_k$ & $|C_k|$ & $ |S_k|$ \\  &&&&\\ \hline

245 &	ICO &	4.219269e+00 &	3.543221e+00 &	791 &	56290\\ 
4143 &	5215 &	4.171512e+00 &	1.753054e+00 &	8 &	990\\ 
4342 &	2295 &	4.875641e+00 &	1.722202e+00 &	4 &	98\\ 
4609 &	13165 &	4.837103e+00 &	1.683919e+00 &	3 &	47\\ 
5296 &	5 &	3.880840e+00 &	1.592507e+00 &	2 &	379\\ 
5804 &	19647 &	4.840889e+00 &	1.528352e+00 &	1 &	1\\ 
6038 &	5429 &	4.778529e+00 &	1.502450e+00 &	11 &	160\\ 
6521 &	3552 &	4.812691e+00 &	1.450600e+00 &	2 &	66\\ 
6910 &	3 &	3.763385e+00 &	1.408780e+00 &	3 &	4252\\ 
7482 &	4 &	3.429287e+00 &	1.356882e+00 &	1 &	1316\\ 
7659 &	9087 &	4.686362e+00 &	1.338609e+00 &	7 &	87\\ 
7675 &	4065 &	4.952008e+00 &	1.337357e+00 &	1 &	33\\ 
7823 &	407 &	4.864080e+00 &	1.325427e+00 &	2 &	45\\ 
8010 &	587 &	4.210932e+00 &	1.309406e+00 &	2 &	288\\ 
8231 &	4305 &	4.793179e+00 &	1.289844e+00 &	2 &	79\\ 
8451 &	26615 &	4.595571e+00 &	1.270736e+00 &	1 &	11\\ 
8498 &	32036 &	4.929577e+00 &	1.266567e+00 &	3 &	25\\ 
8693 &	55024 &	4.718209e+00 &	1.251325e+00 &	1 &	1\\ 
9464 &	19633 &	4.823696e+00 &	1.192426e+00 &	1 &	7\\ 
10136 &	12536 &	3.921984e+00 &	1.145851e+00 &	1 &	12\\ 
10833 &	1787 &	4.823821e+00 &	1.100551e+00 &	1 &	12\\ 
10999 &	43115 &	4.866834e+00 &	1.089582e+00 &	1 &	1\\ 
11355 &	61403 &	4.845165e+00 &	1.063956e+00 &	2 &	3\\ 
11750 &	3551 &	3.830233e+00 &	1.039356e+00 &	7 &	680\\ 
11961 &	2052 &	3.913145e+00 &	1.026976e+00 &	4 &	1758\\ 
12917 &	3624 &	4.715649e+00 &	9.728158e-01 &	2 &	6\\ 
14327 &	59098 &	4.457973e+00 &	8.977427e-01 &	1 &	1\\ 
16694 &	47464 &	4.617778e+00 &	7.857551e-01 &	2 &	3\\ 
19098 &	5074 &	4.788369e+00 &	6.897715e-01 &	2 &	9\\ 
20834 &	16468 &	4.918987e+00 &	6.277720e-01 &	1 &	1\\ 
22168 &	3190 &	4.871516e+00 &	5.837694e-01 &	3 &	3\\ 
22544 &	28583 &	4.775726e+00 &	5.725636e-01 &	2 &	6\\ 
24715 &	10735 &	3.882704e+00 &	5.094120e-01 &	1 &	1\\ 
24967 &	3191 &	3.652424e+00 &	5.030321e-01 &	1 &	27\\ 
25642 &	22585 &	4.686717e+00 &	4.854972e-01 &	1 &	1\\ 
27507 &	6119 &	4.665559e+00 &	4.395695e-01 &	1 &	1\\ 
27508 &	7135 &	4.779884e+00 &	4.395485e-01 &	1 &	3\\ 
27907 &	11388 &	4.783819e+00 &	4.305184e-01 &	1 &	1\\ 
29151 &	16976 &	4.575356e+00 &	4.012544e-01 &	1 &	2\\ 
29477 &	5029 &	3.631126e+00 &	3.941418e-01 &	1 &	3\\ 
31771 &	12970 &	4.339322e+00 &	3.465081e-01 &	2 &	2\\ 
32961 &	16916 &	4.449395e+00 &	3.233989e-01 &	1 &	1\\ 
34118 &	6129 &	4.579289e+00 &	3.027489e-01 &	1 &	1\\ 
35518 &	15156 &	4.969974e+00 &	2.795456e-01 &	1 &	2\\ 
38928 &	2 &	2.399623e+00 &	2.292819e-01 &	1 &	21\\ 
39857 &	5334 &	4.362974e+00 &	2.169966e-01 &	1 &	1\\ 
39872 &	3863 &	3.308673e+00 &	2.167377e-01 &	4 &	25\\ 
40647 &	3261 &	4.595683e+00 &	2.063268e-01 &	1 &	3\\ 
42847 &	16435 &	4.329995e+00 &	1.793789e-01 &	1 &	1\\ 
44417 &	7820 &	4.176016e+00 &	1.617021e-01 &	1 &	4\\ 
45846 &	18074 &	4.159332e+00 &	1.463683e-01 &	1 &	1\\ 
47271 &	26132 &	4.109060e+00 &	1.321669e-01 &	2 &	4\\ 
50106 &	658 &	4.116532e+00 &	1.061901e-01 &	1 &	9\\ 
54440 &	12347 &	4.885824e+00 &	7.330549e-02 &	1 &	1\\ 
58491 &	4586 &	4.861846e+00 &	4.776641e-02 &	1 &	3\\ 
59154 &	9102 &	4.999152e+00 &	4.385418e-02 &	1 &	1\\ 
59683 &	3181 &	4.566481e+00 &	4.115680e-02 &	1 &	11\\ 
61752 &	11817 &	4.893556e+00 &	3.039805e-02 &	1 &	1\\ 
64175 &	420 &	3.007773e+00 &	2.020888e-02 &	1 &	3\\ 
69069 &	3179 &	4.922959e+00 &	4.457720e-03 &	1 &	1\\ 
\hline
    \end{tabular}
    }
\caption{The data for the truncated and factored $\lj38$ network. }
\end{center}
\end{table}
The size distribution of Freidlin's cycles is
presented in Table 3. Naturally, $C_0\equiv C({\rm FCC})$ contains all 71887 states. The second largest 
Freidlin's cycle is $C({\rm ICO})$ containing 791 states. 
The third largest cycle with 45 states corresponds to the third deepest minimum (minimum 16) (Fig. \ref{fcc_ico}).
Note that $C(16)\subset C({\rm ICO})\subset S_{245}$.
About 84\% of Freidlin's cycles $C_k$  consist of single states.
This is the result of the fact that the states in the $\lj38$ network are separated by relatively high barriers.
Therefore, one cannot significantly factor the dynamics of the 
$\lj38$ network by decomposing it into a disjoint union of Freidlin's cycles.
    \begin{table}
\label{table3}
\begin{center}
\scalebox{0.7}{
    \begin{tabular}{ |c |c |}
    \hline \\
    $N$ & The \# of states with $|C(i)|=N$ \\  \hline
     71887 & 1\\
     791 & 1\\
     45 & 1\\
     31 & 1\\
     23 & 1\\
     20 & 1\\
    19 & 2\\
     18 & 1\\
    17 & 2\\
     16 & 3\\
     15 & 9\\
    14 & 9\\
    13 & 7\\
    12 & 16\\
     11 & 12\\
    10 & 34\\
    9 & 41\\
    8 & 79\\
    7 & 132\\
 6 & 228\\
     5 & 389\\
    4& 843\\
    3 & 2108\\
    2 & 6990\\
     1 & 60973\\
\hline
    \end{tabular}
    }
\caption{The distribution of sizes of Freidlin's cycles $C(i)$, $i\in S$ for the $\lj38$ network.}
\end{center}
\end{table}

Now we return to the question whether the Freidlin's cycle $C_{{\rm ICO}}$ can be viewed as a metastable set at the range of temperatures 
$0<T<0.12$ (the solid-solid phase transition critical temperature is $T=0.12$).
The definition given by Bovier in \cite{bovier1} and adjusted to our notations and terminology sounds as follows.
\begin{definition}
\label{def:met}
A Markov process defined on a network with the set of states $S$ is metastable with respect to the subset  $\mathcal{M}\subset S$,
if 
\begin{equation}
\label{met}
\frac{\inf_{s\in \mathcal{M}}\mathbb{E}_{s}[\tau_{ \mathcal{M}\backslash s}]}
{\sup_{i\notin \mathcal{M}}\mathbb{E}_{i}[\tau_{\mathcal{M}}]}\ge \frac{1}{\rho}\gg 1,
\end{equation}
where $\mathbb{E}_j[\tau_A]$ denotes the expected hitting time of the subset $A\subset S$ for the process starting at a state $j$. 
\end{definition}
The states in $\mathcal{M}$ are representative states of metastable sets.
Definition \ref{def:met} treats metastability as a way to factor the dynamics. It says that
a system is metastable if one can find a subset of states $\mathcal{M}$ 
such that the expected time to reach from any state in $\mathcal{M}$  another state in $\mathcal{M}$
is much larger than the expected time to reach from any state not in $\mathcal{M}$ one of the states in $\mathcal{M}$.
We remark that the set $\mathcal{M}$ can be chosen to be the subset of sinks $\{s^{\ast}_k\}_{k=1}^{K}$.
In our case, if Eq. \eqref{met} holds then there exists a spectral gap
$$0<\lambda_1\le\ldots\le\lambda_{K-1}\ll\lambda_{K}\le\ldots\le \lambda_{n-1}.$$
Apparently, there is no significant spectral gap for the $\lj38$ network near $\lambda_{245}$ unless $T<0.0036$, i.e., extremely low.
Therefore, the full $\lj38$ network\footnotemark[1] with 71887 states and infinite observation time is not metastable in the sense of 
the definition of Bovier and collaborators unless the temperature is extremely low.

\footnotetext[1]{Actually, Wales's group created a more complete $\lj38$ network with over a million of local minima. 
Only its part containing  the lowest 10$^5$ local minima is available at \cite{wales_network},
but it is sufficient for modeling the low-temperature dynamics.
}

Now let us look just at 
the numbers $\Delta_k$ corresponding to the states belonging to the Freidlin's cycle $C({\rm ICO})$.
They are protted separately versus $k$ in  Fig. \ref{fig:dico}. The gap between $\Delta({\rm ICO})\equiv\Delta_{245}$ and the second largest $\Delta$
which is $\Delta(264)\equiv\Delta_{1379}$ is more than 1. 
This fact encourages us to consider the definition of metastability introduced by Schuette and collaborators
in the context of general diffusion processes \cite{schuette03,schuette04}.
Their definition relates metastability with ergodicity.
Adjusted for stochastic networks with detailed balance it becomes
\begin{definition}
\label{def:met2}
Let $s\in S$ be a state of a stochastic network with pairwise rates of the form of Eq. \eqref{eq1}. 
The Freidlin's cycle $C(s)$ (the largest Freidlin's cycle containing $s$ and not containing any state with a smaller potential value)
is metastable with exit rate $\lambda(s)$ if for any state $i\in C(s)\backslash\{s\}$  the exit rate $\lambda(i)$ from the Freidlin's cycle $C(i)$
satisfies
\begin{equation}
\label{met2}
\lambda(i)\gg\lambda(s).
\end{equation}
\end{definition}
The graph in Fig. \ref{fig:dico} shows eloquently that the Freidlin's cycle $C({\rm ICO})$
is metastable in the sense of Definition \ref{def:met2}.
\begin{figure}[htbp]
\begin{center}
\centerline{
\includegraphics[width=\textwidth]{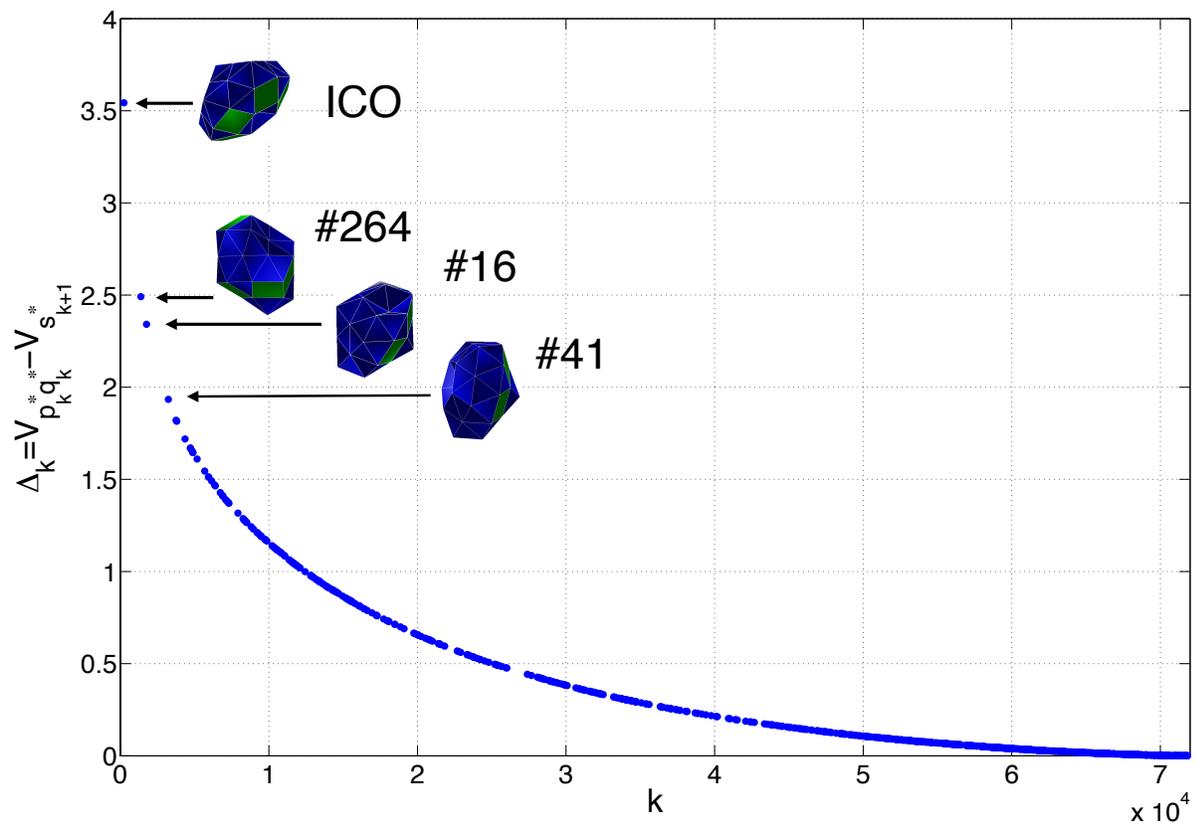}
}
\caption{
The numbers $\Delta_k:=V_{p^{\ast}_{k}q^{\ast}_{k}}-V_{s^{\ast}_{k+1}}$ 
versus $k$ corresponding to the states belonging to the Freidlin's cycle $C({\rm ICO})$.
}
\label{fig:dico}
\end{center}
\end{figure}
In order to visualize the structure of the metastable state $C({\rm ICO})$ we have extracted all of the sinks 
(ordered according to the magnitude of the corresponding eigenvalue)
lying in $C({\rm ICO})$ and plotted a disconnectivity graph for the fisrt 20 of them. 
We have also included the sink corresponding to FCC
(see Fig. \ref{fig:dgraph7}). 
The first four sinks in this substructure are ICO, minimum 264 in Wales's list \cite{wales_network},
the third lowest minimum (minimum 16), and minimum 41. 
These four minima correspond to those eigenvalues of the reduced $\lj38$ network separated by spectral gaps from the rest.
It is apparent from the disconnectivity graph that minimum 264 is separated from ICO by almost as high
barrier as the one separating ICO and FCC.  Freidlin's cycle $C(264)$ consists of 4 states.
\begin{figure}[htbp]
\begin{center}
\centerline{
\includegraphics[width=\textwidth]{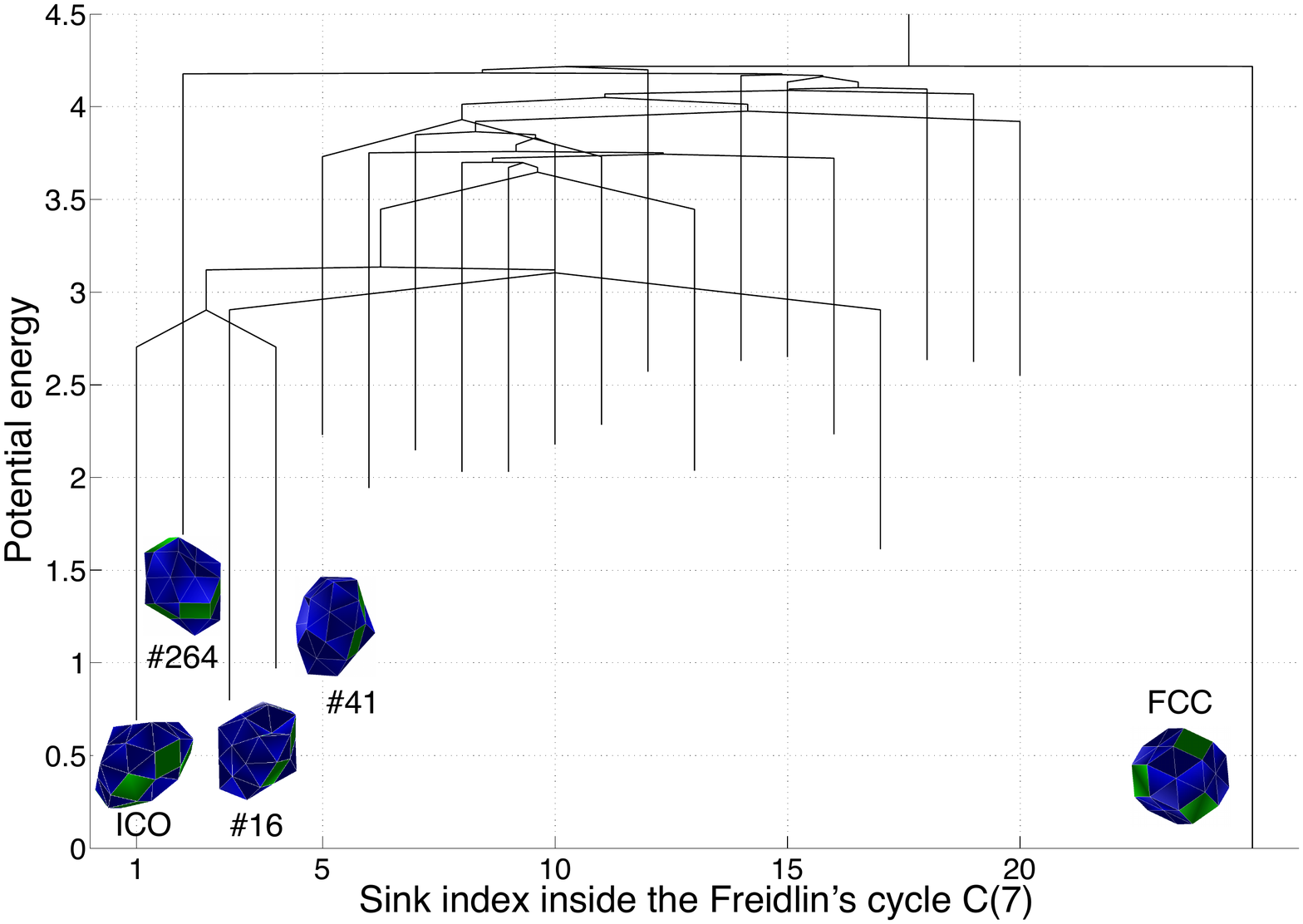}
}
\caption{
The disconnectivity graph showing the first 20 sinks belonging to the Freidlin's cycle $C({\rm ICO})$.
The states are ordered in the increasing order of the sink inside  $C({\rm ICO})$. The potential energy is shown relative to FCC.
}
\label{fig:dgraph7}
\end{center}
\end{figure}

Finally, we perform one more experiment with the $\lj38$ network.
Instead of lumping together states constituting disjoint sets $S_k$ 
and the putting a cap on the highest admissible potential barrier, we
simply truncate the $\lj38$ network without any lumping.
Exactly,
we remove all edges $(i,j)$ with $V_{ij}> 6.0+V_{\rm FCC}$ 
and take the connected component of the resulting network containing FCC and ICO.
It consists of 30520 states and 71750 edges.
This  cut off is equivalent to limiting the observation time.
The graph of the first 100 $\Delta_k$ is shown in Fig. \ref{fig:delta6}. 
There are notable gaps in $\Delta$'s. These differences are 
$\Delta_1-\Delta_2\approx0.19$, $\Delta_2-\Delta_3\approx0.46$, $\Delta_3-\Delta_4\approx0.11$, and $\Delta_4-\Delta_5\approx0.15$.
The other differences are significanly smaller. 
The first eigenvalue $\lambda_1$ is smaller than $\lambda_2$ by the factor of at least 10 if 
the temperature $T< 0.083$.
All four first eigenvalues are separated by gaps 
of at least of the factor of 10 if the temperature $T<0.047$. 
Therefore, the truncated $\lj38$ network is metastable with respect to ICO and FCC  in the sense of Definition \ref{def:met} if $T<0.083$.
It is metastable in the sense of Definition \ref{def:met} with respect to five metastable points, FCC, ICO, 
and the ones corresponding to $\Delta_2$, $\Delta_3$ and $\Delta_4$ in Fig. \ref{fig:delta6} 
(minima 223, 21450 and 7583 in Wales's list \cite{wales_network})
if $T<0.047$.
The disconnectivity graph showing the first 101 sinks of the reduced $\lj38$ network is shown in Fig. \ref{fig:dgraph6}.
\begin{figure}[htbp]
\begin{center}
\centerline{
\includegraphics[width=\textwidth]{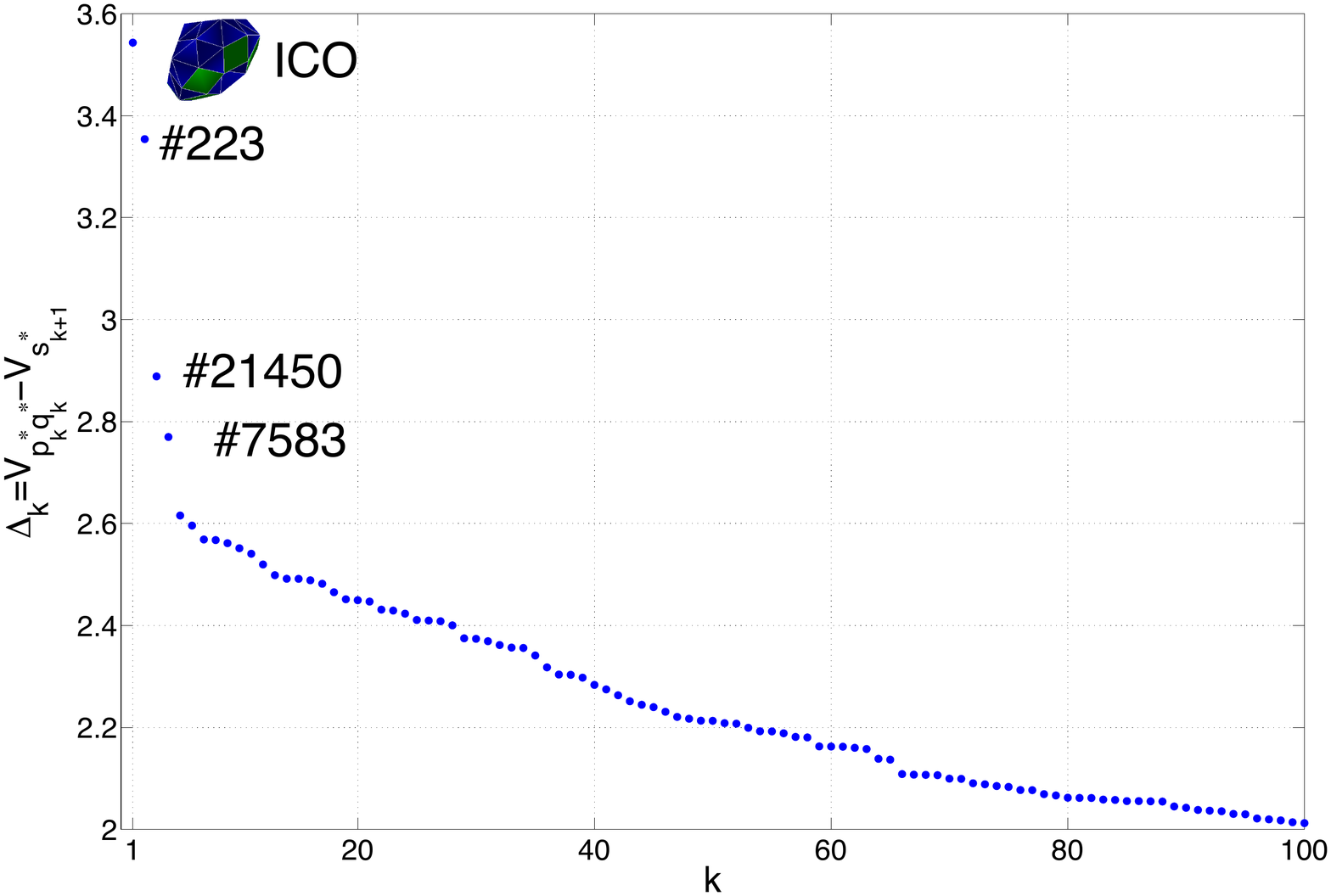}
}
\caption{
The numbers $\Delta_k:=V_{p^{\ast}_{k}q^{\ast}_{k}}-V_{s^{\ast}_{k+1}}$ 
versus $k$ for the reduced $\lj38$ network. Only the first 100 $\Delta_k$'s are shown.
}
\label{fig:delta6}
\end{center}
\end{figure}
\begin{figure}[htbp]
\begin{center}
\centerline{
\includegraphics[width=\textwidth]{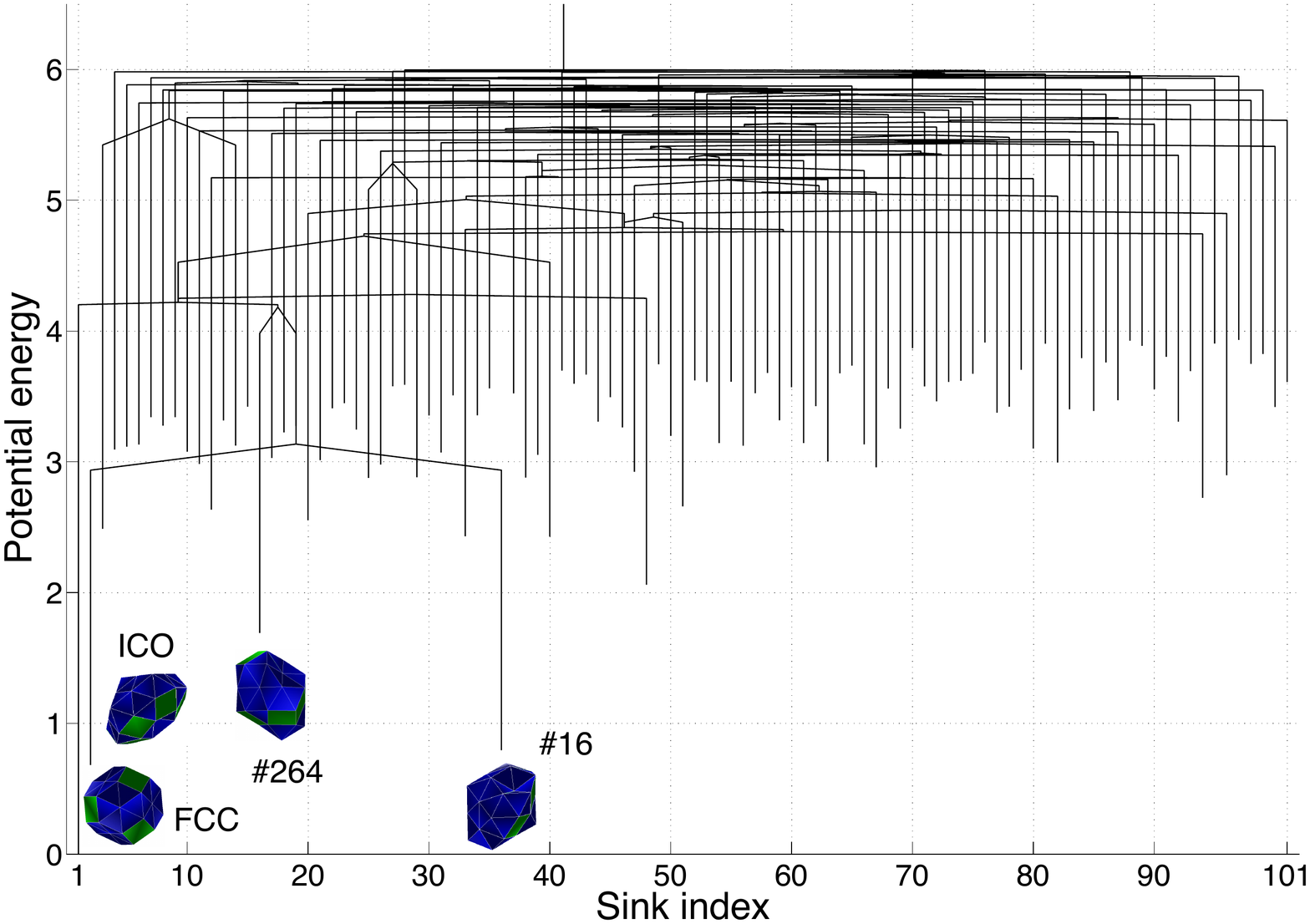}
}
\caption{
The disconnectivity graph showing the first 101 sinks of the truncated $\lj38$ network.
The states are ordered according to the
number of sink that they represent in the truncated network. The potential energy is shown relative to FCC.
}
\label{fig:dgraph6}
\end{center}
\end{figure}

%%%%%%%%%%%%%%%%%%%%%%%%%%%%%%
%%%%%%%%%%%%%%%%%%%%%%%%%%%%%%
\section{Conclusion}
\label{sec:conclusion}
In this work we have considered stochastic networks representing potential energy landscapes.
We have established a connection between the optimal $W$-graphs 
determining the asymptotics of the eigenvalues \cite{wentzell1,wentzell2,f-w} 
and the minimum spanning tree for the edge cost equal to the potential at the corresponding saddle.
We have proven the nested property of the optimal forests
corresponding to the optimal $W$-graphs, i.e., $\mst_{k+1}\subset\mst_k$, $k=1,2,\ldots,n$, and
established recurrence relationships allowing us to construct the optimal forests 
and calculate the asymptotics for the eigenvalues and the eigenvectors.
We have reconciled  Wentzell's formulas, the optimal $W$-graphs, 
Freidlin's cycles and sharp estimates for the low lying spectra by Bovier and collaborators in our construction.

Relying on our theoretical results (Theorems \ref{the2} -\ref{the4}), we have proposed an efficient algorithm  
for computing the asymptotic spectrum starting from the smallest eigenvalues in the absolute value. 
In the nutshell, this algorithm is  a procedure for cutting the minimum spanning tree in a certain order. It is extremely robust and
suitable for complex networks with large numbers of states and edges that do not have to possess any special structural properties
other than the genericness assumption (Assumption \ref{as1}).

We have applied this algorithm to Wales's Lennard-Jones-38  network \cite{wales_network}.
Since the energy landscape of the $\lj38$ has a double-funnel structure, one could expect that 
the $\lj38$ network should have a spectral gap  separating the eigenvalue corresponding to the transition from
the larger and shallower icosahedral funnel to the deeper and narrower face-centered cubic funnel  from the rest.
However, our results demonstrate that this is not the case for the full $\lj38$ network available at \cite{wales_network}.
The aforementioned eigenvalue has number $245$ in the ordered list and it is not separated from the rest by a notable spectral gap.
On the other hand, 
the sinks corresponding to the smallest eigenvalues are essentially irrelevant to the low temperature dynamics.
If the system is initially at the global minimum FCC, the temperature is low, and the observation time is not very large, these high-lying states will 
be extremely unlikely to observe during an experiment or a simulation. 
Putting a cap on the highest barrier separating states from FCC which is equivalent to limiting the observation time
and/or lumping together sets of states, we can obtain a notable spectral gap.
Furthermore, without any capping or lumping, Freidlin's cycle $C({\rm ICO})$  is metastable according to Definition \ref{def:met2} related to ergodicity.

Spectral analysis suggests a way to factor the network-in-hand. We have demonstrated how this can be done for the $\lj38$ network.
The decomposition of the network into disjoint sets $S_k$ (whose indicator functions are subset of the asymptotic eigenvectors)
is helpful for simplification and visualization of low-temperature dynamics.
It also might  be helpful for comparison with experiment, a problem that we leave for the future.

%%%%%%%%%%%%%%%%%%%%%%%%%%%%%%
%%%%%%%%%%%%%%%%%%%%%%%%%%%%%%
\section{Acknowledgements}
I am grateful to Prof. E. Vanden-Eijnden for making me interested in the spectral problem 
and for valuable discussions.
I thank Prof. M. Freidlin for a valuable discussion.
I thank Prof. D. Wales for referring me to his $\lj38$ network and valuable discussion.
This work is supported by the DARPA YFA Grant N66001-12-1-4220 and
the NSF grant  1217118.

%%%%%%%%%%%%%%%%%%%%%%%%%%%%%%%%%%%
%  \renewcommand{\theequation}{A-\arabic{equation}}
%  % redefine the command that creates the equation no.
%  \setcounter{equation}{0}  % reset counter 
%  \section*{APPENDIX}  % use *-form to suppress numbering

\appendix
\section{Proof of Lemma \ref{lemma1}}
\begin{proof}
First we prove Claim $(ii)$.
We will proceed from converse. Let us assume that the edge $(p_1^{\ast},q_1^{\ast})$ belongs to the optimal graph $g^{\ast}_k$ for some $k\in\{3,\ldots,n-1\}$.
Then one can replace $(p_1^{\ast},q_1^{\ast})$ with another edge $(p,q)$ not in $g^{\ast}_k$  and possibly pick another sink 
so that the sum over the edges and sinks  in Eq. \eqref{vk2}  decreases. I.e., if $g_k$ is the $W$-graph obtained as a result of these replacements,
and $\mathcal{T}_k$ and $W_k$ are the corresponding tree and the set of sinks of $g_k$, then
$$\sum_{(i,j)\in \mathcal{T}_k} V_{ij}+\sum_{i\in W_k} V_i<\sum_{(i,j)\in \mst_k} V_{ij}+\sum_{i\in W^{\ast}_k} V_i.$$

There is no single recipe for the choice of the edge $(p,q)$. 
We will have to consider several cases. 
Let  $w^{\ast}_{12}:=w^{\ast}(s_1^{\ast},s_2^{\ast})$ 
be the unique path in the minimum spanning tree
$\mst$  connecting the sinks $s_1^{\ast}$ and $s_2^{\ast}$ of the optimal $W$-graph $g^{\ast}_2$. 
The edge  $(p_1^{\ast},q_1^{\ast})$ must belong to  $w^{\ast}_{12}$,  as $s^{\ast}_1$ and $s_2^{\ast}$ 
belong to different connected components of $g^{\ast}_2$. Without the loss of generality we assume that 
$$w^{\ast}_{12}=\{s^{\ast}_1,~\ldots,p_1^{\ast},~q_1^{\ast},~\ldots,~s^{\ast}_2\}.$$
We observe that 
\begin{equation}
\label{vmax}
V_{p_1^{\ast}q_1^{\ast}}=\max_{(p,q)\in w^{\ast}_{12}}V_{pq},
\end{equation}
as otherwise we get a contradiction with Eq. \eqref{pqs2}.
By Assumption \ref{as1} the maximum in Eq. \eqref{vmax} is reached at the unique edge $(p_1^{\ast},q_1^{\ast})$.

Further we will need the following definition. Let us consider the $W$-graph $\hat{g}$ obtained from $g^{\ast}_k$ by
removing the edge $(p_1^{\ast},q_1^{\ast})$ and adjusting the directions of the edges so that the  sink of each connected component of $\hat{g}$ 
is the state with minimal potential in it. Then for any state $i$,  $sink(i)$ is the sink of the connected component of $\hat{g}$ 
containing $i$. 
We will consider five cases:
\begin{description}
\item[Case A] All edges of the path $w^{\ast}_{12}$ belong to $g^{\ast}_k$.
\item[Case B] There is an edge in $w^{\ast}_{12}$ not belonging to $g^{\ast}_k$.
\begin{description}
\item[Case B.1] $V_{sink(p_1^{\ast}) }\le V_{sink(q_1^{\ast}) }$
\begin{description}
\item[Case B.1.1] There is an edge $(p,q)\in w^{\ast}(q^{\ast}_1,s_2^{\ast})\subset w^{\ast}_{12}$ such that $(p,q)\notin \mst_k$.
\item[Case B.1.2] The whole path $w^{\ast}(q^{\ast}_1,s_2^{\ast})$ belongs to $g^{\ast}_k$.
\end{description}
\item[Case B.2] $V_{sink(p_1^{\ast}) }> V_{sink(q_1^{\ast}) }$.
\begin{description}
\item[Case B.2.1] There is an edge $(p,q)\in w^{\ast}(s^{\ast}_1,p_1^{\ast})\subset w^{\ast}_{12}$ such that $(p,q)\notin \mst_k$.
\item[Case B.2.2] The whole path $w^{\ast}(s^{\ast}_1,p_1^{\ast})$ belongs to $\mst_k$.
\end{description}
\end{description}
\end{description}
Cases A, B.1.1, B.1.2, and B.2.1 are illustrated in Fig. \ref{fig:the3}. Case B.2.2 is impossible.
Indeed, if the whole path  $w^{\ast}(s^{\ast}_1,p_1^{\ast})$ belongs to  $g^{\ast}_k$ then the states  $p_1^{\ast}$ and $s_1^{\ast}$
belong to the same connected component of $g_k^{\ast}$. Hence $sink(p_1^{\ast})  = s_1$, the state with the minimal potential in the whole network. 
This contradicts to the assumption that $V_{sink(p_1^{\ast}) }>V_{sink(q_1^{\ast}) }$. 
\begin{figure}[htbp]
\begin{center}
\centerline{
(a)\includegraphics[width=\textwidth]{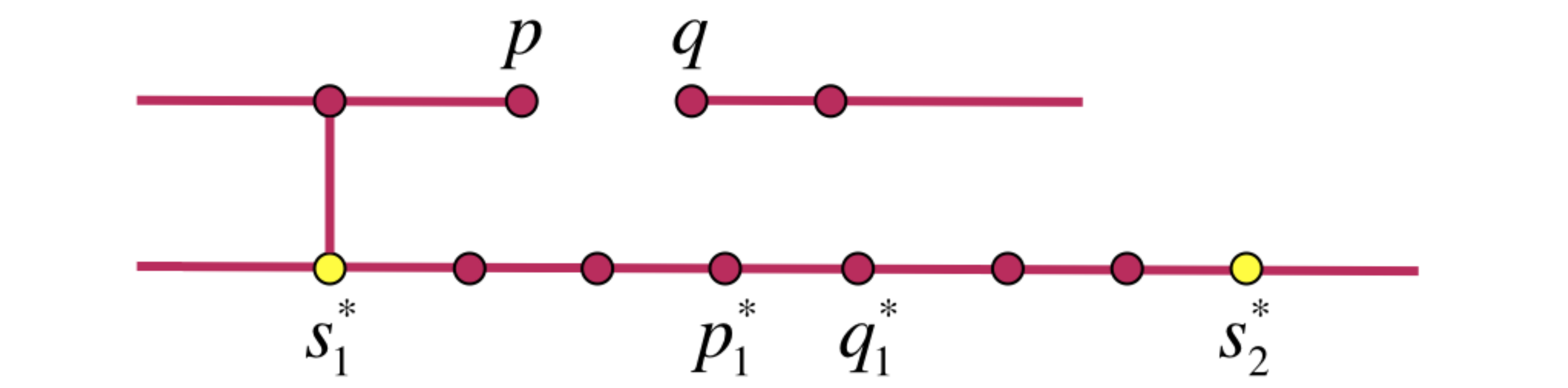}}
\vspace{5mm}
\centerline{
(b)\includegraphics[width=\textwidth]{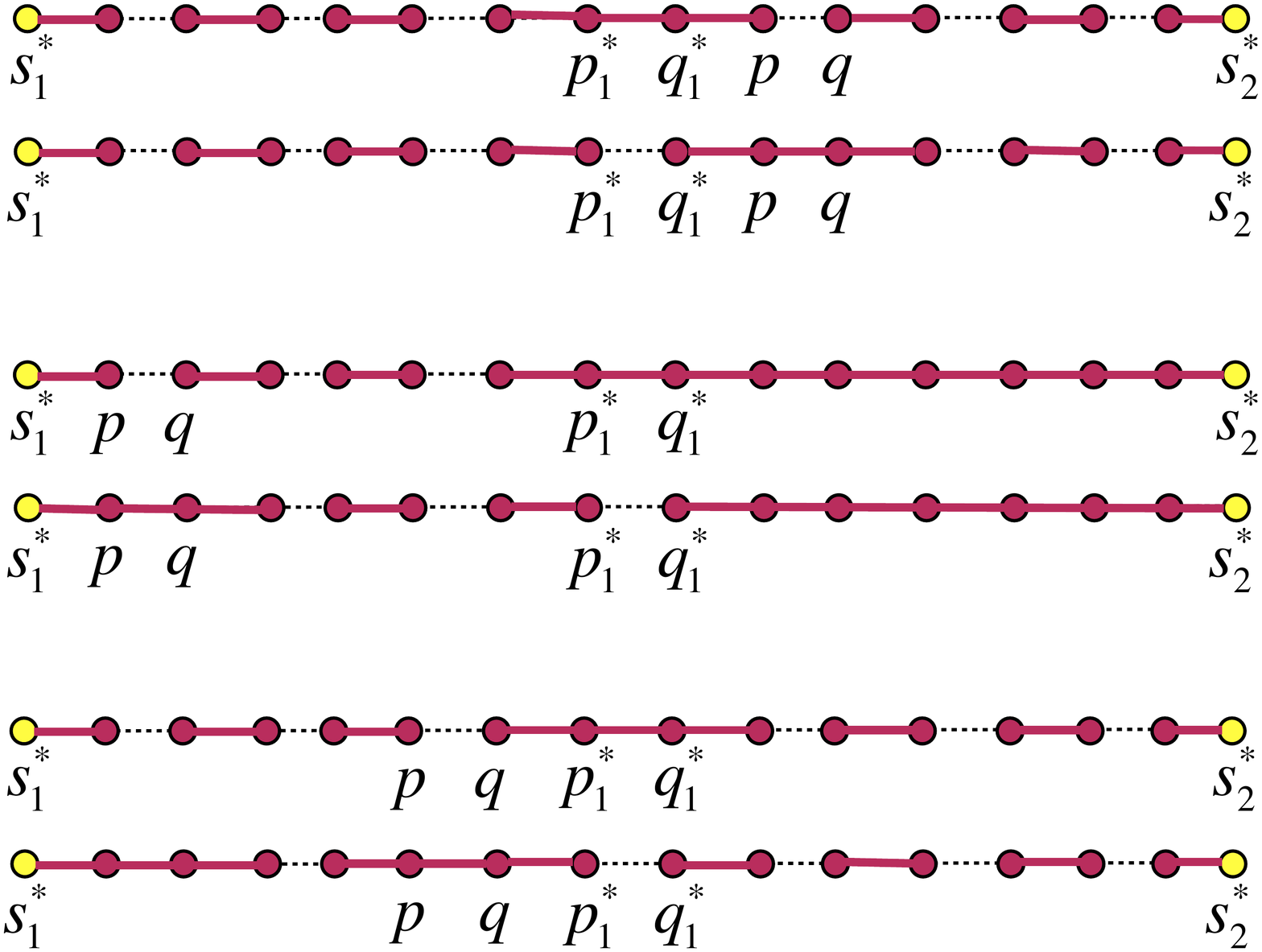}}
\caption{Illustration for the proof of Theorem \ref{the3}. (a): Case A. (b): Case B. Top: Case B.1.1. Center: Case B.1.2. Bottom: Case B.2.1 }
\label{fig:the3}
\end{center}
\end{figure}

Now we will explain how to choose the edge $(p,q)$ in each of the cases A, B.1.1, B.1.2, and B.2.1.
\begin{description}
\item[Case A] {
Since the $W$-graph $g^{\ast}_k$ is not connected, there is an edge  $(p,q)\in\mst$ such that $(p,q)\notin g^{\ast}_k$ 
and $p$ belongs to the connected component of $g^{\ast}_k$ containing the path $w^{\ast}_{12}$ (see Fig. \ref{fig:the3}(a)). 
Replacing the edge $(p_1^{\ast},q_1^{\ast})$ with the edge $(p,q)$
and choosing $s_2^{\ast}$ to be $sink(q_1^{\ast})$,
we transform the optimal $W$-graph $g^{\ast}_k$ into another $W$-graph $g_k$. By the assumption that $g^{\ast}_k$ is optimal we have
\begin{equation}
 V_{p_1^{\ast}q_1^{\ast}}+V_{sink(q)}-(V_{pq}+V_{s^{\ast}_2})<0,\quad{\rm i.e.}\quad  
V_{p_1^{\ast}q_1^{\ast}}-V_{s^{\ast}_2} < V_{pq}-V_{sink(q)}. 
\end{equation}
The inequalities above are strict by Assumption \ref{as1}.
This contradicts to the definition of $(p_1^{\ast},q_1^{\ast})$  and $s^{\ast}_2$ given by Eq. \eqref{pqs2}.
Hence the $W$-graph $g^{\ast}_k$ is not optimal.
}
\item[Case B.1.1]{
In this case, there is an edge $(p,q)\in w^{\ast}(q^{\ast}_1,s_2^{\ast})$ such that $(p,q)\notin g^{\ast}_k$ 
but $w^{\ast}(q_1^{\ast},p)\in g^{\ast}_k$ (see Fig. \ref{fig:the3}(b), Top). 
Replacing the edge $(p_1^{\ast},q_1^{\ast})$ with the edge $(p,q)$
and properly choosing sinks, 
we transform the optimal graph $g^{\ast}_k$ into another $W$-graph $g_k$. 
By the assumption that $g^{\ast}_k$ is optimal we have
\begin{align}
&V_{p_1^{\ast}q_1^{\ast}}+\min\{V_{sink(p_1^{\ast}) },V_{sink(q_1^{\ast}) }\} + V_{sink(q)} \notag\\
& -(V_{pq}+V_{sink(p_1^{\ast}) }+\min\{V_{sink(q_1^{\ast}) },V_{sink(q)}\})\le 0.
\end{align}
By assumption, $V_{sink(p_1^{\ast}) }\le V_{sink(q_1^{\ast}) }$. Hence $\min\{V_{sink(p_1^{\ast}) },V_{sink(q_1^{\ast}) }\} = V_{sink(p_1^{\ast}) }$. Therefore,
\begin{equation}
V_{p_1^{\ast}q_1^{\ast}}+V_{sink(q)} < V_{pq}+\min\{V_{sink(q_1^{\ast}) },V_{sink(q)}\}). 
\end{equation}
The inequality above is strict by Assumption \ref{as1}.
Noting that $V_{sink(q)}\ge \min\{V_{sink(q_1^{\ast}) },V_{sink(q)}\}$
we conclude that $V_{p_1^{\ast}q_1^{\ast}}<V_{pq}$. 
This contradicts to the fact that that $V_{p_1^{\ast}q_1^{\ast}}=\max_{(i,j)\in w^{\ast}_{12}}V_{ij}$ (see Eq. \eqref{vmax}).
Hence $g^{\ast}_k$ is not optimal.

}
\item[Case B.1.2]{
In this case, there is an edge $(p,q)\in w^{\ast}(s^{\ast}_1,p_1^{\ast})$ such that $(p,q)\notin g^{\ast}_k$ but $w(a,p)\in g^{\ast}_k$ (see Fig. \ref{fig:the3}(b), Center). 
Replacing the edge $(p_1^{\ast},q_1^{\ast})$ with the edge $(p,q)$
and choosing $s_2^{\ast}$ to be  $sink(q_1^{\ast})$, 
we transform the optimal graph $g^{\ast}_k$ into another $W$-graph $g_k$. 
By the assumption that $g^{\ast}_k$ is optimal we have
\begin{equation}
V_{p_1^{\ast}q_1^{\ast}}+V_{sink(q)} +\min\{V_{sink(p_1^{\ast}) },V_{sink(q_1^{\ast}) }\} 
 - (V_{pq}+V_{sink(p_1^{\ast}) } +V_{s^{\ast}_2})\le 0.
\end{equation}
By assumption, $V_{sink(p_1^{\ast}) }\le V_{sink(q_1^{\ast}) }$. Hence $\min\{V_{sink(p_1^{\ast}) },V_{sink(q_1^{\ast}) }\} = V_{sink(p_1^{\ast}) }$. 
Therefore,
\begin{equation}
V_{p_1^{\ast}q_1^{\ast}} - V_{s^{\ast}_2} <  V_{pq}+V_{sink( q)}. 
\end{equation}
The inequality above is strict by Assumption \ref{as1}.
This contradicts to the definition of $(p_1^{\ast},q_1^{\ast})$  and $s^{\ast}_2$ given by Eq. \eqref{pqs2}.
Hence $g^{\ast}_k$ is not optimal.
}
\item[Case B.2.1]{
In this case, there is an edge $(p,q)\in w^{\ast}(s^{\ast}_1,p_1^{\ast})$ such that $(p,q)\notin g^{\ast}_k$ but $w^{\ast}(q,p_1^{\ast})\in g^{\ast}_k$ 
(see Fig. \ref{fig:the3}(b), Bottom). 
Replacing the edge $(p_1^{\ast},q_1^{\ast})$ with the edge $(p,q)$
and properly choosing sinks, 
we transform the optimal graph $g^{\ast}_k$ into another $W$-graph $g_k$. 
By the assumption that $g^{\ast}_k$ is optimal we have
\begin{align}
& V_{p_1^{\ast}q_1^{\ast}}+V_{sink(p )} +\min\{V_{sink(p_1^{\ast}) },V_{sink(q_1^{\ast}) }\} \notag \\
& - (V_{pq}+V_{sink(q_1^{\ast}) } +\min\{V_{sink(p )},V_{sink(p_1^{\ast}) }\})\le 0.
\end{align}
By assumption, $V_{sink(p_1^{\ast}) }> V_{sink(q_1^{\ast}) }$, hence $\min\{V_{sink(p_1^{\ast}) },V_{sink(q_1^{\ast}) }\} = V_{sink(q_1^{\ast}) }$. Therefore,
\begin{equation}
V_{p_1^{\ast}q_1^{\ast}}+V_{sink(p )} < V_{pq}+\min\{V_{sink(p )},V_{sink(p_1^{\ast}) }\}. 
\end{equation}
The inequality above is strict by Assumption \ref{as1}.
Noting that $V_{sink(p )}\ge \min\{V_{sink(p_1^{\ast}) },V_{sink(p )}\}$
we conclude that $V_{p_1^{\ast}q_1^{\ast}}<V_{pq}$. 
This contradicts to the fact that that $V_{p_1^{\ast}q_1^{\ast}}=\max_{(i,j)\in w^{\ast}_{12}}V_{ij}$ (see Eq. \eqref{vmax}).
Hence $g^{\ast}_k$ is not optimal.
}
\end{description}
Now we prove Claim $(iii)$. 
Since the edge $(p_1^{\ast},q_1^{\ast})$ does not belong to $g^{\ast}_k$, $k=2,3,\ldots,n$, the states $s^{\ast}_1$ and $s^{\ast}_2$ 
belong to different connected components of the graphs $g^{\ast}_k$, $k=2,3,\ldots,n$. 
By Observation \ref{o1}, the state $s^{\ast}_2$ has the smallest value of the potential in its connected component of $g^{\ast}_2$.
Since the connected components of $g^{\ast}_k$, $k=3,\ldots,n$
containing the state $s^{\ast}_2$ are subgraphs of the of the connected component of $g^{\ast}_2$ containing $s^{\ast}_2$,
$s^{\ast}_2$ has also the smallest value of the potential in its connected components of $g^{\ast}_k$, $k=3,\ldots,n$. Therefore, it must be a sink 
of $g^{\ast}_k$, $k=3,\ldots,n$. 
\end{proof}





\end{document}